\documentclass[11pt]{article}
\usepackage{latexsym,amsfonts,amssymb,amsmath,amsthm}

\usepackage{color,comment}

\begin{document}

\newtheorem{thm}{Theorem}[section]
\newtheorem{cor}{Corollary}[section]
\newtheorem{lem}{Lemma}[section]
\newtheorem{prop}{Proposition}[section]
\newtheorem{defn}{Definition}[section]
\newtheorem{rk}{Remark}[section]
\newtheorem{nota}{Notation}[section]
\newtheorem{Ex}{Example}[section]
\def\nm{\noalign{\medskip}}

\numberwithin{equation}{section}

\newcommand{\ds}{\displaystyle}
\newcommand{\pf}{\medskip \noindent {\sl Proof}. ~ }
\newcommand{\p}{\partial}
\renewcommand{\a}{\alpha}
\newcommand{\z}{\zeta}
\newcommand{\pd}[2]{\frac {\p #1}{\p #2}}
\newcommand{\norm}[1]{\left\| #1 \right \|}
\newcommand{\dbar}{\overline \p}
\newcommand{\eqnref}[1]{(\ref {#1})}
\newcommand{\na}{\nabla}
\newcommand{\Om}{\Omega}
\newcommand{\ep}{\epsilon}
\newcommand{\tmu}{\widetilde \epsilon}
\newcommand{\vep}{\varepsilon}
\newcommand{\tlambda}{\widetilde \lambda}
\newcommand{\tnu}{\widetilde \nu}
\newcommand{\vp}{\varphi}
\newcommand{\RR}{\mathbb{R}}
\newcommand{\CC}{\mathbb{C}}
\newcommand{\NN}{\mathbb{N}}
\renewcommand{\div}{\mbox{div}~}
\newcommand{\bu}{{\bf u}}
\newcommand{\la}{\langle}
\newcommand{\ra}{\rangle}
\newcommand{\Scal}{\mathcal{S}}
\newcommand{\Lcal}{\mathcal{L}}
\newcommand{\Kcal}{\mathcal{K}}
\newcommand{\Dcal}{\mathcal{D}}
\newcommand{\tScal}{\widetilde{\mathcal{S}}}
\newcommand{\tKcal}{\widetilde{\mathcal{K}}}
\newcommand{\Pcal}{\mathcal{P}}
\newcommand{\Qcal}{\mathcal{Q}}
\newcommand{\id}{\mbox{Id}}
\newcommand{\stint}{\int_{-T}^T{\int_0^1}}

\newcommand{\be}{\begin{equation}}
\newcommand{\ee}{\end{equation}}

\newcommand{\rd}{{\mathbb R^d}}
\newcommand{\rr}{{\mathbb R}}
\newcommand{\alert}[1]{\fbox{#1}}
\newcommand{\eqd}{\sim}
\def\R{{\mathbb R}}
\def\N{{\mathbb N}}
\def\Q{{\mathbb Q}}
\def\C{{\mathbb C}}
\def\ZZ{{\mathbb Z}}
\def\l{{\langle}}
\def\r{\rangle}
\def\t{\tau}
\def\k{\kappa}
\def\a{\alpha}
\def\la{\lambda}
\def\De{\Delta}
\def\de{\delta}
\def\ga{\gamma}
\def\Ga{\Gamma}
\def\ep{\varepsilon}
\def\eps{\varepsilon}
\def\si{\sigma}
\def\Re {{\rm Re}\,}
\def\Im {{\rm Im}\,}
\def\E{{\mathbb E}}
\def\P{{\mathbb P}}
\def\Z{{\mathbb Z}}
\def\D{{\mathbb D}}
\def\p{\partial}
\newcommand{\ceil}[1]{\lceil{#1}\rceil}

\title{Logistic type attraction-repulsion chemotaxis systems with  a free boundary or unbounded boundary.
II. Spreading-vanishing dichotomy in a domain with a free boundary}

\author{Lianzhang Bao\thanks{School of Mathematics, Jilin University, Changchun, 130012, P. R. China, and Department of Mathematics and Statistics,
Auburn University,  AL 36849, U. S. A. (lzbao@jlu.edu.cn).}\,\,   and
Wenxian Shen \thanks{Department of Mathematics and Statistics,
Auburn University,
 AL 36849, U. S. A. (wenxish@auburn.edu), partially supported by the NSF grant DMS--1645673.}}

\date{}

\maketitle

\begin{abstract}
The current series of research papers is to investigate the asymptotic dynamics in logistic type chemotaxis models in one space dimension with a free boundary or unbounded boundary.
Such a model with a free boundary  describes the spreading of a new or invasive species subject to the influence of some chemical substances
in an environment with a free boundary representing the spreading front.
In this first of the series,  we  investigated the dynamical behaviors of logistic type chemotaxis models on the half line $\mathbb{R}^+$, which are formally corresponding limit systems of the free boundary problems.
In the second of the series, we  establish the spreading-vanishing dichotomy in  chemoattraction-repulsion systems with a free boundary as well as with double free boundaries.
\end{abstract}


\medskip


\section{Introduction}
The current series of research papers is  to  study  the spreading and vanishing dynamics of
 the following attraction-repulsion chemotaxis system with a  free boundary and time and space dependent logistic source,
\begin{equation}\label{one-free-boundary-eq}
\begin{cases}
u_t = u_{xx} -\chi_1  (u  v_{1,x})_x + \chi_2 (u v_{2,x})_x + u(a(t,x) - b(t,x)u), \quad 0<x<h(t)
\\
 0 = \partial_{xx} v_1 - \lambda _1v_1 + \mu_1u,  \quad  0<x<h(t)
 \\
 0 = \partial_{xx}v_2 - \lambda_2 v_2 + \mu_2u,  \quad 0<x<h(t)
 \\
 h'(t) = -\nu u_x(t,h(t))
\\
u_x(t,0) = v_{1,x}(t,0) = v_{2,x}(t,0) = 0
\\
 u(t,h(t)) = v_{1,x}(t,h(t)) = v_{2,x}(t,h(t)) = 0
 \\
 h(0) = h_0,\quad u(x,0) = u_0(x),\quad  0\leq x\leq h_0,
\end{cases}
\end{equation}
and to study the
asymptotic dynamics of
\begin{equation}
\label{half-line-eq1}
\begin{cases}
u_t = u_{xx} -\chi_1  (u  v_{1,x})_x +\chi_2(u v_{2,x})_x+ u(a(t,x) - b(t,x)u),\quad x\in (0,\infty)
\cr
 0 = v_{1,xx} - \lambda_1v_1 + \mu_1u,  \quad x\in (0,\infty)\cr
 0=v_{2,xx}-\lambda_2 v_2+\mu_2 u,  \quad x\in (0,\infty)\cr
u_x(t,0)=v_{1,x}(t,0)=v_{2,x}(t,0)=0,
\end{cases}
\end{equation}
where $\nu>0$ in \eqref{one-free-boundary-eq} is a positive constant,  and in both \eqref{one-free-boundary-eq} and \eqref{half-line-eq1},   $\chi_i$, $\lambda_i$, and $\mu_i$ ($i=1,2$)  are nonnegative constants,   and $a(t,x)$ and $b(t,x)$  satisfy the following assumption,

\medskip

\noindent {\bf (H0)} {\it $a(t,x)$ and $b(t,x)$ are bounded $C^1$ functions on $\RR\times [0,\infty)$,
and
$$a_{\inf}:=\inf_{t\in\RR,x\in [0,\infty)}a(t,x)>0,\quad b_{\inf}:=\inf_{t\in\RR,x\in[0,\infty)}b(t,x)>0.
$$
}

\medskip

Biological backgrounds of \eqref{one-free-boundary-eq} and \eqref{half-line-eq1} are discussed in the first part of the series (\cite{BaoShen1}) and \cite{luca2003chemotactic}. The free boundary condition in \eqref{one-free-boundary-eq} is also derived in \cite{BaoShen1} based on the assumption that, as the expanding front propagates, the population suffers a loss of constant units per unit volume at the front, and that, near the propagating front, the
   population density is  close to zero.
Formally, \eqref{half-line-eq1} can be viewed as the limit system of \eqref{one-free-boundary-eq} as $h(t)\to \infty$.

The objective of this series is to investigate the asymptotic dynamics of \eqref{half-line-eq1} and
  the spreading and vanishing scenario in \eqref{one-free-boundary-eq}. In the first part of the series (\cite{BaoShen1}), we studied the asymptotic dynamics of \eqref{half-line-eq1}.  In this second part of the series, we study the spreading and vanishing scenario in \eqref{one-free-boundary-eq}.
  To state the main results of the current paper, we first recall some results proved in \cite{BaoShen1}.

Let
$$
C_{\rm unif}^b(\RR^+)=\{u\in C(\RR^+)\,|\, u(x)\,\, \text{is uniformly continuous and bounded on}\,\, \RR^+\}
$$
with norm $\|u\|_\infty=\sup_{x\in\RR^+}|u(x)|$, and
$$
C_{\rm unif}^b(\RR)=\{u\in C(\RR)\,|\, u(x)\,\, \text{is uniformly continuous and bounded on}\,\, \RR\}
$$
with norm $\|u\|_\infty=\sup_{x\in\RR}|u(x)|$. Define
\begin{align}\label{m-eq}
M= \min\Big\{ &\frac{1}{\lambda_2}\big( (\chi_2\mu_2\lambda_2-\chi_1\mu_1\lambda_1)_+ + \chi_1\mu_1(\lambda_1-\lambda_2)_+ \big),\nonumber\\
& \qquad \frac{1}{\lambda_1}\big( (\chi_2\mu_2\lambda_2-\chi_1\mu_1\lambda_1)_+ + \chi_2\mu_2(\lambda_1-\lambda_2)_+ \big) \Big\}
 \end{align}
 and
 \begin{align}\label{k-eq}
K=\min\Big\{&\frac{1}{\lambda_2}\Big(|\chi_1\mu_1\lambda_1-\chi_2\mu_2\lambda_2|+\chi_1\mu_1|\lambda_1-\lambda_2|\Big),\nonumber\\
&\quad  \frac{1}{\lambda_1}\Big(|\chi_1\mu_1\lambda_1-\chi_2\mu_2\lambda_2|+\chi_2\mu_2|\lambda_1-\lambda_2|\Big) \Big\}.
\end{align}
  Let {\bf (H1)}- {\bf (H3)}  be the following standing assumptions.

\medskip

\noindent {\bf (H1)}  $b_{\inf}>\chi_1\mu_1-\chi_2\mu_2+M$.

\medskip

\noindent {\bf (H2)} $b_{\inf}>\Big(1+\frac{a_{\sup}}{a_{\inf}}\Big)\chi_1\mu_1-\chi_2\mu_2+M$.

\medskip

\noindent {\bf (H3)}  $b_{\inf}>\chi_1\mu_1-\chi_2\mu_2+K$.

\medskip

Note  that
$$
M\le \chi_2\mu_2.
$$
Hence $b_{\inf}\ge \chi_1\mu_1$ implies {\bf (H1)}. In the case $\chi_2=0$, we can choose $\lambda_2=\lambda_1$,  and then
$M=0$ and $K=\chi_1\mu_1$. Hence {\bf (H1)} becomes
$b_{\inf}>\chi_1\mu_1$, {\bf (H2)} becomes $b_{\inf}>(1+\frac{a_{\sup}}{a_{\inf}})\chi_1\mu_1$, and {\bf (H3)} becomes $b_{\inf}>2\chi_1\mu_1$. In the case $\chi_1=0$, we can also choose $\lambda_1=\lambda_2$, and then $M=\chi_2\mu_2$ and $K=\chi_2\mu_2$. Hence {\bf (H1)}
(resp.{\bf (H2)}, {\bf (H3)}) becomes $b_{\inf}>0$.
Biologically, {\bf (H1), (H2)}, and {\bf (H3)} indicate that the chemo-attraction sensitivity is relatively small with respect to logistic damping.

\smallskip

When {\bf (H1)} holds, we put
\begin{equation}
\label{M0-eq}
M_0=\frac{a_{\sup}}{b_{\inf}+\chi_2\mu_2-\chi_1\mu_1-M}
\end{equation}
and
\begin{equation}
\label{m0-eq}
m_0= \frac{a_{\inf}\big(b_{\inf}-(1+\frac{a_{\sup}}{a_{\inf}})\chi_1\mu_1+\chi_2\mu_2-M\big)}{(b_{\inf}-\chi_1\mu_1+\chi_2\mu_2-M)
(b_{\sup}-\chi_1\mu_1+\chi_2\mu_2)}.
\end{equation}
Note that if {\bf (H2)} holds, then $m_0>0$.

Among those, we proved the following results in \cite{BaoShen1}.

\begin{thm}
\label{half-line-thm}
Consider \eqref{half-line-eq1}.
\begin{itemize}
\item[(1)] (Global existence) If {\bf (H1)} holds, then for any $t_0\in\RR$ and any nonnegative function $u_0\in C^{b}_{\rm unif}(\RR^+)$, \eqref{half-line-eq1} has a unique solution
    $(u(t,x;t_0,u_0),v_1(t,x;t_0,u_0)$, $v_2(t,x;t_0,u_0))$ with $u(t_0,x;t_0,u_0)=u_0(x)$ defined for $t\ge t_0$. Moreover,
    $$
0\le u(t,x;t_0,u_0)\le  \max\{\|u_0\|_\infty, M_0\}\quad \forall \,\, t\in [t_0,\infty), \,\, x\in [0,\infty)
$$
(see  \cite[Theorem 1.1]{BaoShen1}),
and
$$
\limsup_{t\to\infty}\sup_{x\in\RR^+} u(t,x;t_0,u_0)\le M_0
$$
(see \cite[Lemma 2.3]{BaoShen1}).

\item[(2)] (Persistence, \cite[Theorem 1.2]{BaoShen1})
\begin{itemize}

 \item[(i)]
 If {\bf (H1)} holds, then for any $u_0\in C_{\rm unif}^b(\RR^+)$ with $\inf_{x\in\RR^+}u_0(x)>0$, there is $m(u_0)>0$ such that
 $$
 m(u_0)\le u(t,x;t_0,u_0)\le \max\{\|u_0\|_\infty,M_0\}\quad \forall\,\, t\ge t_0,\,\, x\in\RR^+.
 $$

 \item[(ii)] If {\bf (H2)} holds, then, for any $u_0\in C_{\rm unif}^b(\RR^+)$ {with $\inf_{x\in\RR^+}u_0(x)>0$,} there is $T(u_0)>0$ such that
 $$
 m_0\le u(t,x;t_0,u_0)\le M_0+1\quad \forall \, t_0\in\RR\,\, t\ge t_0+T(u_0),\,\, x\in\RR^+.
 $$
\end{itemize}

\item[(3)] (Existence of strictly positive entire solution, \cite[Theorem 1.3(1)]{BaoShen1}) If {\bf (H1)} holds, then \eqref{half-line-eq1}
admits a strictly positive entire solution $(u^*(t,x),v_1^*(t,x),v_2^*(t,x))$ (i.e., $(u^*(t,x),v_1^*(t,x)$, $v_2^*(t,x))$ is defined for all
$t\in\RR$ and $\inf_{t\in\RR,x\in\RR^+}u^*(t,x)>0$). Moreover, if
 $a(t+T,x)\equiv a(t,x)$ and $b(t+T,x)\equiv b(t,x)$, then \eqref{half-line-eq1}
admits a strictly positive $T$-periodic solution $ (u^*(t,x)$, $v_1^*(t,x)$, $v_2^*(t,x))= (u^*(t+T,x),v_1^*(t+T,x),v_2^*(t+T,x)) $.

\item[(4)] (Stability and uniqueness of strictly positive entire solution, \cite[Theorem 1.3(2)]{BaoShen1})

\begin{itemize}

\item[(i)]  Assume {\bf (H3)} and $a(t,x)\equiv a(t)$ and $b(t,x)\equiv b(t)$, then \eqref{half-line-eq1}
 admits a unique strictly positive entire solution $(u^*(t,x),v_1^*(t,x),v_2^*(t,x))$, and
 for any $u_0\in C_{\rm unif}^b(\RR^+)$ with $\inf_{x\in \RR}u_0(x)>0$,
   \begin{equation}\label{Eq-asymptotic-0}
 \lim_{t\to\infty} \| u(t+t_0,\cdot;t_0,u_0) - u^*(t+t_0,\cdot)\|_\infty =0, \forall\,  t_0\in \mathbb{R}.
\end{equation}

\item[(ii)] Assume {\bf (H3)}.
There are $\chi_1^*>0$ and $\chi_2^*>0$ such that, if $0\le \chi_1\le \chi_1^*$ and $0\le \chi_2\le \chi_2^*$, then
 \eqref{half-line-eq1}
 admits a unique strictly positive entire solution $(u^*(t,x),v_1^*(t,x),v_2^*(t,x))$, and
for any $u_0\in C_{\rm unif}^b(\RR^+)$ with $\inf_{x\in \RR^+}u_0(x)>0$,
\eqref{Eq-asymptotic-0} holds.
\end{itemize}
\end{itemize}
\end{thm}

We remark that the above theorem applies to all the limit equations of \eqref{half-line-eq1}. To be more precise, let
$$
H(a,b)={\rm cl}\{ (a(t+\cdot,\cdot),b(t+\cdot,\cdot))|t\in\RR\}
$$
with open compact topology, where the closure is taken under the open compact topology.
For any $(\tilde a,\tilde b)\in H(a,b)$, consider
\begin{equation}
\label{half-line-eq2}
\begin{cases}
u_t = u_{xx} -\chi_1  (u  v_{1,x})_x +\chi_2(u v_{2,x})_x+ u(\tilde a(t,x) -\tilde  b(t,x)u),\quad x\in (0,\infty)
\cr
 0 = v_{1,xx} - \lambda_1v_1 + \mu_1u,  \quad x\in (0,\infty)\cr
 0=v_{2,xx}-\lambda_2 v_2+\mu_2 u,  \quad x\in (0,\infty)\cr
u_x(t,0)=v_{1,x}(t,0)=v_{2,x}(t,0)=0.
\end{cases}
\end{equation}
Then Theorem \ref{half-line-thm} also holds for \eqref{half-line-eq2}.

The main results of the current paper are stated in the following theorems.
The first theorem is on the global existence of nonnegative solutions of \eqref{one-free-boundary-eq}.

 \begin{thm}[Global existence]
 \label{free-boundary-thm1}
  If {\bf (H1)} holds, then for any $t_0\in\RR$, and any $h_0>0$ and any function $u_0(x)$ on $[0,h_0]$ satisfying
  \begin{equation}
  \label{initial-cond-eq}
  u_0\in C^2[0,h_0],\quad u_0(x)\ge 0\,\, {\rm for}\,\, x\in [0,h_0],\quad {\rm and}\,\,  u_0'(0)=0,\,\,  u_0(h_0)=0,
  \end{equation}
   \eqref{one-free-boundary-eq} has a unique globally defined solution $(u(t,x;t_0,u_0,h_0)$, $v_1(t,x;t_0,u_0,h_0)$, $v_2(t,x;t_0,u_0,h_0)$, $h(t;t_0,u_0,h_0))$
with $u(t_0,x;t_0,u_0,h_0)=u_0(x)$ and $h(t_0;t_0,u_0,h_0)=h_0$. Moreover,
\begin{equation}\label{thm1-h-bound}
{ 0\leq h'(t) \leq 2\nu M_1 C_0},
\end{equation}
 \begin{equation}
 \label{thm1-eq1}
0\le u(t,x;t_0,u_0,h_0)\le  \max\{\|u_0\|_\infty, M_0\}\quad \forall \,\, t\in [t_0,\infty), \,\, x\in [0,h(t;t_0,u_0,h_0)),
\end{equation}
and
\begin{equation}
\label{thm1-eq2}
\limsup_{t\to\infty} \sup_{x\in [0,h(t;t_0,u_0,h_0))}u(t_0+t,x;t_0,u_0,h_0)\le M_0,
\end{equation}
where $M_1$ is a big enough constant and $C_0 = \max\{\|u_0\|_{\infty}, M_0\}$.
\end{thm}

\smallskip

Assume {\bf (H1)}. For any given $t_0\in\RR$,  and any given $h_0>0$ and  $u_0(\cdot)$ satisfying
\eqref{initial-cond-eq}, by the nonnegativity of $u(t,x;t_0,u_0,h_0)$, $h^{'}(t;t_0,u_0,h_0)\ge 0$ for all $t>t_0$. Hence
$\lim_{t\to\infty}h(t;t_0,u_0,h_0)$ exists. Put
$$
h_\infty(t_0,u_0,h_0)=\lim_{t\to\infty} h(t;t_0,u_0,h_0).
$$
We say {\it vanishing} occurs if $h_\infty(t_0,u_0,h_0)<\infty$ and
$$
\lim_{t\to\infty}\|u(t,\cdot;t_0,u_0,h_0)\|_{C([0,h(t;t_0,u_0,h_0)])}=0.
$$
We say {\it spreading} occurs if $h_\infty(t_0,u_0,h_0)=\infty$ and
for any $L>0$, $$\liminf_{t\to\infty} \inf_{0\le x \le L} u(t,x;u_0,h_0)>0.$$

 For given $l>0$,  consider the following linear equation,
\begin{equation}\label{Eq-PLE1}
\begin{cases}
 v_t = v_{xx} + a(t,x) v, \quad 0<x<l\cr
 v_x(t,0) = v(t,l) =0.
\end{cases}
\end{equation}
Let $[\lambda_{\min}(a,l),\lambda_{\max}(a,l)]$ be the principal spectrum interval of \eqref{Eq-PLE1}
(see Definition \ref{spectrum-def1}). Let $l^*>0$ be such that
$\lambda_{\min}(a,l)>0$ for $ l>l^*$ and $\lambda_{\min}(a,l^*)=0$ (see \eqref{l-star-eq} for the existence and uniqueness
of $l^*$).

Our second theorem is about the spreading and vanishing dichotomy scenario in \eqref{one-free-boundary-eq}.

\begin{thm}[Spreading-vanishing dichotomy]
\label{free-boundary-thm2}
Assume that {\bf (H1)} holds.
 For any given $t_0\in\RR$, and $h_0>0$ and $u_0(\cdot)$ satisfying \eqref{initial-cond-eq}, we have that
either

\smallskip

\noindent   (i) vanishing occurs and $h_\infty(t_0,u_0,h_0)\le l^*$; or

\smallskip

\noindent  (ii) spreading occurs.
\end{thm}



\smallskip
For given $t_0\in\RR$,  and  $h_0>0$ and   $u_0(\cdot)$ satisfying
\eqref{initial-cond-eq}, if spreading occurs, it is interesting to know whether {\it local uniform persistence} occurs in the sense that
there is a positive constant $\tilde m_0$ independent of the initial data such that for any $L>0$,
$$\liminf_{t\to\infty}\inf_{0\le x\le L}u(t,x;t_0,u_0,h_0)\ge \tilde m_0,
$$
   and whether {\it local uniform convergence} occurs in the sense that  $\lim_{t\to\infty} u(t,x;t_0,u_0,h_0)$ exists locally uniformly. We have the following theorem
along this direction.

\begin{thm}[Persistence and convergence]
\label{free-boundary-thm3}
Assume that {\bf (H1)} holds and that $h_0>0$ and $u_0(\cdot)$ satisfy \eqref{initial-cond-eq}.

\smallskip

\noindent (i) (Local uniform persistence) For any given $t_0\in\RR$, if  $h_\infty(t_0,u_0,h_0)=\infty$ and {\bf (H2)} holds, then
for any $L>0$, $$\liminf_{t\to\infty} \inf_{0\le x \le L} u(t,x;t_0,u_0,h_0)>m_0,$$
where $m_0$ is as in \eqref{m0-eq}.

\smallskip
\noindent (ii) (Local uniform convergence)  Assume that {\bf (H3)} holds,
and  that for any $(\tilde a,\tilde b)\in H(a,b)$, \eqref{half-line-eq2}
has a  unique strictly positive entire solution $(u^*(t,x;\tilde a,\tilde b),v_1^*(t,x;\tilde a,\tilde b)$, $v_2^*(t,x;\tilde a,\tilde b))$.
Then for  any given $t_0\in\RR$, if
  $h_\infty(t_0,u_0,h_0)=\infty$, there are $\chi_1^* > 0, \chi_2^* > 0$ such that to any {$0\leq \chi_1\leq \chi_1^*$, $0\leq \chi_2\leq \chi_2^*$,}
for any $L>0$,
\begin{equation}\label{FBP-stability}
\lim_{t\to\infty} \sup_{0\le x\le L} |u(t,x;t_0,u_0,h_0)-u^*(t,x;a,b)|=0.
\end{equation}
\smallskip

\noindent (iii) (Local uniform convergence)
Assume that {\bf (H3)} holds, and that $a(t,x)\equiv a(t)$ and $b(t,x)\equiv b(t)$. Then for any given  $t_0\in\RR$, if $h_\infty(t_0,u_0,h_0)=\infty$, then
for any $L>0$,
$$
\lim_{t\to\infty} \sup_{0\le x\le L} |u(t,x;u_0,h_0)-u^*(t)|=0,
$$
where $u^*(t)$ is the unique strictly  positive entire solution of the ODE
\begin{equation}
\label{fisher-kpp-ode}
u^{'}=u(a(t)-b(t)u)
\end{equation}
(see \cite[Lemma 2.5]{Issa-Shen-2017} for the existence and uniqueness of strictly  positive entire solutions of
\eqref{fisher-kpp-ode}).
\end{thm}

We conclude the introduction with the following remarks.

\smallskip

First,  in \cite{Salako2017Global}, an attraction-repulsion chemotaxis system with
constant logistic source $u(a-bu)$  on the whole space is studied. Assuming  $b>\chi_1\mu_1-\chi_2\mu_2+M$, the authors proved the
existence and uniqueness
of globally defined solutions with nonnegative, bounded, and uniformly continuous initial functions (see \cite[Theorem A]{Salako2017Global}).
Theorem \ref{free-boundary-thm1} and Theorem \ref{half-line-thm}(1) are the counterparts of   \cite[Theorem A]{Salako2017Global}
for attraction-repulsion chemotaxis systems with
time dependent  logistic source on environments with a free boundary and on a half space.
Furthermore, the authors of \cite{Salako2017Global} proved the existence, uniqueness, and stability of a strictly positive entire solution
under the assumption $b>\chi_1\mu_1-\chi_2\mu_2+K$ (see \cite[Theorem B]{Salako2017Global}).
Theorem \ref{free-boundary-thm3}(ii), (iii) and Theorem \ref{half-line-thm} (3), (4) are the counterpart
of  \cite[Theorem B]{Salako2017Global} for attraction-repulsion chemotaxis systems with
time dependent  logistic source on environments with a free boundary and on a half space.
Note that the conditions in Theorem \ref{half-line-thm}(4)  (ii) imply the conditions in   Theorem \ref{free-boundary-thm3}(ii).

Second, we became aware of a paper by Zhang et al. \cite{Zhang2018Afree} that addresses a similar models to ours after finishing the paper. They investigated the problem with the equations for the chemoattractant and chemorepulsion being on the whole half space $0<x<\infty$
 and the birth and death damping coefficients in logistic term being constants.

Third, Theorem \ref{free-boundary-thm2}  reveals a spreading-vanishing dichotomy scenario for  attraction-repulsion chemotaxis systems with logistic source and with a free boundary.
It is seen that, as long as the species spreads into a region with size greater than $l^*$, which only depends on $a(\cdot,\cdot)$,  at some finite time, then the species will eventually  spread into the whole space.
  The spreading-vanishing dichotomy scenario for \eqref{one-free-boundary-eq}  observed in Theorem  \ref{free-boundary-thm2}
 is a strikingly different  spreading scenario from the following attraction-repulsion chemotaxis system with logistic source on the whole space,
\begin{equation}
\label{whole-space-eq1}
\begin{cases}
u_t = u_{xx} -\chi_1  (u  v_{1,x})_x +\chi_2(u v_{2,x})_x+ u(a(t,x) - b(t,x)u),\quad x\in \R
\cr
 0 = v_{1,xx} - \lambda_1v_1 + \mu_1u,  \quad x\in \R\cr
 0=v_{2,xx}-\lambda_2 v_2+\mu_2 u,  \quad x\in \R,
\end{cases}
\end{equation}
where $a(t,x)$ and $b(t,x)$ are bounded $C^1$ functions on $\RR\times\RR$ with $a_{\inf}:=\inf_{(t,x)\in\RR\times\RR}a(t,x)>0$ and $b_{\inf}:=\inf_{(t,x)\in\RR\times\RR}b(t,x)>0$.
Assuming {\bf (H1)}, it can be proved that spreading always occurs in \eqref{whole-space-eq1}.
 In fact, it can be proved that, for any $0<c<2\sqrt {a_{\inf}}$ and $u_0\in C_{\rm unif}^b(\R)$ with $u_0\ge 0$ and
$\{x\,|\, u_0(x)>0\}$ being bounded and nonempty, the unique globally defined solution $(u(t,x;u_0), v_1(t,x;u_0),v_2(t,x;u_0))$
of \eqref{whole-space-eq1}  with
$u(0,x;u_0)=u_0(x)$ satisfies
$$
\liminf_{t\to\infty}\inf_{|x|\le ct}u(t,x;u_0)>0
$$
(see \cite[Theorem 1.1]{SaShXu} for the proof in the case that $a(t)$ and $b(t)$ are constants and $\chi_2=0$, the general case can be proved
by the similar arguments of \cite[Theorem 1.1]{SaShXu}).

\smallskip

Fourth, in the absence of chemotaxis (i.e., $\chi_1=\chi_2=0$),  \eqref{one-free-boundary-eq} reduces to
 \begin{equation}\label{one-free-boundary-no-chemotaxis-eq}
\begin{cases}
u_t = u_{xx} + u(a(t,x) - b(t,x)u), \quad 0<x<h(t)
\\
h^{'}=-\nu u_x(t,h(t))
\\
u_x(t,0) = 0
\\
 u(t,h(t)) = 0
 \\
 h(0) = h_0,\quad u(x,0) = u_0(x),\quad  0\leq x\leq h_0,
\end{cases}
\end{equation}
which was first introduced by Du and Lin in \cite{DuLi}
to understand the spreading of species.
It is proved in \cite{DuLi} that \eqref{one-free-boundary-no-chemotaxis-eq} with  $a(t,x)\equiv a$ and  $b(t,x)\equiv b$
 exhibits the following spreading-vanishing dichotomy:  for any given $h_0>0$ and
  either vanishing occurs (i.e. $\lim_{t\to\infty}h(t;u_0,h_0)\le l^*$, where $l^*$ is as in Theorem \ref{free-boundary-thm1}, and $\lim_{t\to\infty}u(t,x;u_0,h_0)=0$) or
spreading occurs (i.e. $\lim_{t\to\infty}h(t;u_0,h_0)=\infty$ and $\lim_{t\to\infty}u(t,x$; $u_0,h_0)=a/b$ locally uniformly in $x\in \RR^+$). The above spreading-vanishing dichotomy has also been extended to the cases with space periodic logistic source, time periodic logistic source, and time almost periodic logistic source, etc. (see \cite{du2013diffusive}, \cite{du2013pulsating}, \cite{Li2016duffusive}, \cite{Li2016duffusive2}, etc.).
Letting $\chi_1=\chi_2=0$, Theorems  \ref{free-boundary-thm2} and \ref{free-boundary-thm3} partially recover the existing results for \eqref{one-free-boundary-no-chemotaxis-eq}.

\smallskip

Finally, the techniques developed for the study of  \eqref{one-free-boundary-eq} can be modified to study the following double spreading fronts free boundary problem,
 \begin{eqnarray}\label{two-free-boundary-eq}
\begin{cases}
u_t = u_{xx} -\chi_1  (u  v_{1,x})_x  + \chi_2 (u v_{2,x})_x + u(a(t,x) - b(t,x) u),& x\in (g(t),h(t))
\\
 0 = (\partial_{xx} - \lambda_1I)v_1 + \mu_1u, & x\in  (g(t),h(t))
 \\
 0 = (\partial_{xx} - \lambda_2I)v_2 + \mu_2u, & x\in  (g(t),h(t))
 \\
g'(t) = -\nu u_x(g(t),t)\\
 h'(t) = -\nu u_x(h(t),t)
\\
u(g(t),t) = v_{1,x}(t,g(t)) = v_{2,x}(t,g(t)) = 0
\\
 u(h(t),t) = v_{1,x}(t,h(t)) = v_{2,x}(t,h(t)) = 0,
\end{cases}
\end{eqnarray}
where $a(t,x)$ and $b(t,x)$ are bounded $C^1$ functions on $\RR\times\RR$ with $a_{\inf}:=\inf_{(t,x)\in\RR\times\RR}a(t,x)>0$ and $b_{\inf}:=\inf_{(t,x)\in\RR\times\RR}b(t,x)>0$.

To be more precise,
assume {\bf (H1)}. For any given $t_0,g_0,h_0\in\RR$ with $g_0<h_0$, and given $u_0(\cdot)$ satisfying
\begin{equation}
u_0\in C^2([g_0,h_0]),\quad u_0(g_0)=u_0(h_0)=0,\quad u_0(x)\ge 0,\quad x\in (g_0,h_0),
\end{equation}
by the similar arguments as those of Theorem \ref{free-boundary-thm1}, \eqref{two-free-boundary-eq} has a unique solution
$(u(t,x;t_0,u_0,g_0,h_0)$, $v_1(t,x;t_0,u_0,g_0,h_0)$, $v_2(t,x;t_0,u_0,g_0,h_0)$, $g(t;t_0,u_0g_0,h_0)$, $h(t;t_0,u_0,g_0,h_0))$
defined on $[t_0,\infty)$ with $u(t_0,x;t_0,u_0,g_0,h_0)=u_0(x)$, $g(t_0;t_0,u_0,g_0,h_0)=g_0$, and $h(t_0;t_0,u_0,g_0,h_0)=h_0$.
By the nonnegativity of $u(t,x;t_0,u_0,g_0,h_0)$, $g^{'}(t;t_0,u_0,g_0,h_0)\le 0$ and
$h^{'}(t;t_0,u_0,g_0,h_0)\ge 0$ for $t>t_0$. Hence the limits
$$g_\infty(t_0,u_0,g_0,h_0):=\lim_{t\to\infty}g(t;t_0,u_0,g_0,h_0),\,\,
h_\infty(t_0,u_0,g_0,h_0):=\lim_{t\to\infty} h(t;t_0,u_0,g_0,h_0)$$
 exist. If no confusion occurs, we may put
 $g(t)=g(t;t_0,u_0,g_0,h_0)$, $h(t)=h(t;t_0,u_0,g_0,h_0)$, $g_\infty=g_\infty(t_0,u_0,g_0,h_0)$ and
$h_\infty=h_\infty(t_0,u_0,g_0,h_0)$.

 For given $l_-<l_+$,  consider the following parabolic equation,
\begin{equation}\label{Eq-PLE2}
\begin{cases}
 v_t = v_{xx} + a(t,x) v, \quad l_-<x<l_+
 \\
 v(t,l_-) = v(t,l_+) =0.
\end{cases}
\end{equation}
 Let $[\lambda_{\min}(a,l_-,l_+),\lambda_{\max}(a,l_-,l_+)]$ be the principal spectrum interval of \eqref{Eq-PLE2}
(see Definition \ref{spectrum-def2}). Let $l^{**}>0$ be such that
$\lambda_{\min}(a,l_-,l_+)>0$ for any $l_-<l_+$ with $ l_+-l_->l^{**}$ and $\lambda_{\min}(a,l_-,l_+)=0$ for some $l_-<l_+$ with $l_+-l_-=l^{**}$ (see \eqref{l-star-star-eq} for the existence and uniqueness
of $l^{**}$).

We have the following theorem on the spreading-vanishing dichotomy for  \eqref{two-free-boundary-eq}.

\begin{thm}(Spreading-vanishing dichotomy with the double free boundaries)
\label{free-boundary-thm4}
Assume that {\bf(H1)} holds. For any given $t_0,g_0,h_0\in\RR$ with $g_0<h_0$ and $u_0(\cdot)\in C^2([g_0,h_0])$ satisfying \eqref{two-free-boundary-eq}, we have
either

\smallskip

\noindent (i) $h_\infty - g_\infty \leq l^{**}$, and $\lim_{t\to\infty} \|u(t,\cdot;u_0,h_0)\|_{C([g(t),h(t)]}=0$;  or

\smallskip

\noindent  (ii) $h_\infty= -g_\infty =\infty$, and for any $L>0$,
$$\liminf_{t\to\infty} \inf_{|x|\le L} u(t,x;u_0,h_0,g_0)>0.$$

\noindent  Moreover, if $h_\infty= -g_\infty =\infty$,{\bf(H3)} holds, and the following system
\begin{eqnarray}\label{whole-space-eq}
\begin{cases}
u_t = u_{xx} -\chi_1  (u  v_{1,x})_x  + \chi_2 (u v_{2,x})_x + u(\tilde a(t,x) - \tilde b(t,x) u),& x\in\RR
\\
 0 = (\partial_{xx} - \lambda_1I)v_1 + \mu_1u, & x\in  \RR
 \\
 0 = (\partial_{xx} - \lambda_2I)v_2 + \mu_2u, & x\in  \RR
\end{cases}
\end{eqnarray}
 has a unique strictly positive entire solution $(u^*(t,x;\tilde a,\tilde b),v_1^*(t,x;\tilde a,\tilde b),v_2^*(t,x;\tilde a,\tilde b))$,
then for any $L>0$,
$$
\lim_{t\to\infty} \sup_{|x|\le L} |u(t,x;u_0,h_0)-u^*(t,x;a,b)|=0.
$$
In particular, if $h_\infty= -g_\infty =\infty$, {\bf (H3)} holds, and $a(t,x)\equiv a(t)$ and $b(t,x)\equiv b(t)$, then
for any $L>0$,
$$
\lim_{t\to\infty} \sup_{0\le x\le L} |u(t,x;u_0,h_0)-u^*(t)|=0,
$$
where $u^*(t)$ is the unique positive entire solution of the ODE
$$
u^{'}=u(a(t)-b(t)u).
$$
\end{thm}

The rest of this paper is organized in the following way. In section 2, we present some preliminary lemmas to be used in the proofs of the main theorems in later sections.
We study  the local and global existence of nonnegative solutions of \eqref{one-free-boundary-eq} and prove Theorem
\ref{free-boundary-thm1} in section 3.
In section 4, we explore  the spreading-vanishing dichotomy  behaviors of \eqref{one-free-boundary-eq} and prove Theorem \ref{free-boundary-thm2}.
We consider the local uniform persistence and local uniform convergence of  \eqref{one-free-boundary-eq} and prove Theorem \ref{free-boundary-thm3}
in section 5.
 In section 6, we study spreading- vanishing dichotomy scenario in  \eqref{two-free-boundary-eq} and prove Theorem \ref{free-boundary-thm4}.

\section{Preliminary}

In this section, we present some preliminary materials to be used in the later sections.

\subsection{Principal spectral theory}

In this subsection, we recall some principal spectrum theory for linear parabolic equations.
We first recall some principal spectrum theory for the following linear parabolic equation,
\begin{equation}
\label{linear-eq1}
\begin{cases}
u_t=u_{xx}+a(t,x)u,\quad 0<x<l\cr
u_x(t,0)=u(t,l)=0,
\end{cases}
\end{equation}
where $l>0$ is a given positive number and $a\in C^1([0,\infty))$ is bounded and $a_{\inf}>0$. Let
$$X(l)=\{u\in C([0,l])\,|\, u_x(0)=u(l)=0\}$$
with norm $\|u\|_{X(l)}=\sup_{x\in [0,l]}|u(x)|$,
$$
X^+(l)=\{u\in X(l)\,|\, u(x)\ge 0\,\, {\rm for}\,\, 0\le x\le l\},
$$
and
$$
X^{++}(l)=\{u\in X(l)\,|\, u(x)>0\,\, {\rm for}\,\, 0\le x<l,\,\, u_x(l)<0\}.
$$
Let $U(t,s;a,l)$ be the evolution operator of \eqref{linear-eq1} on $X$, that is, for any
$u_0\in X(l)$,
$$(U(t,s;a,l)u_0)(x):=u(t,x;s,u_0),$$
where $u(t,x;s,u_0)$ is the solution of \eqref{linear-eq1} with $u(s,x;s,u_0)=u_0(x)$.

\begin{defn}
\label{spectrum-def1}
$[\lambda_{\min}(a,l),\lambda_{\max}(a,l)]$ is called the {\rm principal spectrum interval} of \eqref{linear-eq1}, where
$$
\lambda_{\min}(a,l)=\liminf_{t-s\to \infty}\frac{ \ln \|U(t,s;a,l)\|}{t-s}, \quad \lambda_{\max}(a,l)=\limsup_{t-s\to\infty}\frac{\ln \|U(t,s;a,l)\|}{t-s}.
$$
\end{defn}

Let
$$
H(a)={\rm cl}\{a\cdot t(\cdot,\cdot):=a(t+\cdot,\cdot)\,|\, t\in \RR\}
$$
with open compact topology, where the closure is taken under the open compact topology. In literature, $H(a)$ is called the hull of $a(\cdot,\cdot)$.

\begin{lem}
\label{spectrum-lm1}
$[\lambda_{\min}(a,l),\lambda_{\max}(a,l)]$ is a compact interval. Moreover,
there is a continuous function $\omega:H(a)\to X^{++}(l)$ such that
$$
\|w(\tilde a)\|_{X(l)}=1\quad \forall\,\, \tilde a\in H(a),
$$
$$
\frac{U(t,0;\tilde a,l)\omega(b)}{\|U(t,0;\tilde a,l)\omega(b)\|_{X(l)}}=\omega(b\cdot t)\quad \forall\, t>0,\,\, \tilde a\in H(a),
$$
and
$$
\lambda_{\min}(a,l)=\liminf_{t-s\to\infty} \frac{\ln \|U(t-s,0;a\cdot s,l)\omega(a\cdot s)\|}{t-s},
$$
$$
 \lambda_{\max}(a,l)=\limsup_{t-s\to\infty} \frac{\ln \|U(t-s,0;a\cdot s,l)\omega(a\cdot s)\|}{t-s}.
 $$
\end{lem}

\begin{proof}
It follows from \cite[Propositions 2.3 and 2.4]{MiSh1} and \cite[Proposition 4.1.9]{MiSh2}.
\end{proof}

\begin{lem}
\label{spectrum-lm2}
$\lambda_{min}(a,l)$ and $\lambda_{\max}(a,l)$ are monotone in $a$ and $l$, that is, for any $a_1(\cdot,\cdot)\le a_2(\cdot,\cdot)$ and
$0<l_1\le l_2$,
$$
\lambda_{\min}(a_1,l_1)\le \lambda_{\min}(a_2,l_2),\quad \lambda_{\max}(a_1,l_1)\le \lambda_{\max}(a_2,l_2).
$$
\end{lem}

\begin{proof}
By comparison principle for parabolic equations,
$$
(U(t-s,0;a_1\cdot s,l_1)\omega(a_1\cdot s))(x)\le (U(t-s,0;a_2\cdot s,l_2)\tilde \omega(a_2\cdot s))(x)\quad \forall\,\, t\ge s,\,\, 0\le x\le l_1,
$$
where $\tilde \omega(a_2\cdot s)(x)=\omega(a_1\cdot s)(x)$ for $0\le x\le l_1$ and $\tilde \omega(a_2\cdot s)(x)=0$ for
$l_1\le x\le l_2$.
The lemma then follows from Lemma \ref{spectrum-lm1}.
\end{proof}

\begin{rk}
\label{spectrum-rk1}
\begin{itemize}
\item[(1)] If $a(t,x)\equiv a$, then $\lambda_{\min}(a,l)=\lambda_{\max}(a,l)$ and $\lambda(a,l):=\lambda_{\min}(a,l)$
is the principal eigenvalue of
$$
\begin{cases}
u_{xx}+au=\lambda u,\quad 0<x<l\cr
u_x(0)=u(l)=0.
\end{cases}
$$
In this case, $\lambda(a,l)=-\frac{\pi^2}{2l^2}+a$ and $\omega(a)=\cos\frac{\pi x}{2l}$. Therefore,
$\lim_{l\to 0+}\lambda(a,l)=-\infty$ and $\lim_{l\to\infty}\lambda(a,l)=a$.

\item[(2)] If $a(t+T,x)=a(t,x)$, then $\lambda_{\min}(a,l)=\lambda_{\max}(a,l)$ and $\lambda(a,l):=\lambda_{\min}(a,l)$ is the principal eigenvalue of
$$
\begin{cases}
-u_t+u_{xx}+a(t,x)u=\lambda u,\quad 0<x<l\cr
u(t+T,x)=u(t,x),\quad 0\le x\le l\cr
u_x(t,0)=u(t,l)=0.
\end{cases}
$$
\end{itemize}
\end{rk}

\begin{lem}
\label{spectrum-lm3}
$\lim_{l\to 0+} \lambda_{\max}(a,l)=-\infty$ and $\lim_{l\to \infty} \lambda_{\min}(a,l)\ge a_{\inf}(>0)$.
\end{lem}

\begin{proof}
By Lemma \ref{spectrum-lm2},
$$
\lambda_{\min}(a,l)\ge \lambda(a_{\inf},l),\quad \lambda_{\max}(a,l)\le \lambda(a_{\sup},l).
$$
The lemma then follows from Remark \ref{spectrum-rk1}(1).
\end{proof}

By Lemmas \ref{spectrum-lm2} and  \ref{spectrum-lm3}, there is a unique $l^*>0$ such that
\begin{equation}
\label{l-star-eq}
\lambda_{\min}(a,l)>0\quad \forall\,\, l>l^*,\quad \lambda_{\min}(a,l^*)=0.
\end{equation}

Next, we recall some principal spectrum theory for
 the following parabolic equation,
\begin{equation}\label{linear-eq2}
\begin{cases}
 v_t = v_{xx} + a(t,x) v, \quad l_-<x<l_+
 \\
 v(t,l_-) = v(t,l_+) =0,
\end{cases}
\end{equation}
where $l_-<l_+$ and $a\in C^1(\RR)$ is bounded and $a_{\inf}>0$.
Let
$$
X(l_-,l_+)=\{v\in C([l_-,l_+])\,|\, v(l_-)=v(l_+)=0\}
$$
with norm $\|v\|=\sup_{x\in [l_-,l_+]}|v(x)|$. Let $U(t,s;a,l_-,l_+)$ be the evolution operator of \eqref{linear-eq2} on $X(l_-,l_+)$.

\begin{defn}
\label{spectrum-def2}
$[\lambda_{\min}(a,l_-,l_+),\lambda_{\max}(a,l_-,l_+)]$ is called the {\rm principal spectrum interval} of \eqref{linear-eq2}, where
$$
\lambda_{\min}(a,l_-,l_+)=\liminf_{t-s\to \infty}\frac{ \ln \|U(t,s;a,l_-,l_+)\|}{t-s}, \quad \lambda_{\max}(a,l_-,l_+)=\limsup_{t-s\to\infty}\frac{\ln \|U(t,s;a,l_-,l_+)\|}{t-s}.
$$
\end{defn}

Similarly, we have

\begin{lem}
\label{spectrum-lm4}
\begin{itemize}
\item[(1)] $[\lambda_{\min}(a,l_-,l_+),\lambda_{\max}(a,l_-,l_+)]$ is a compact interval.

\item[(2)] If $a_1(t,x)\le a_2(t,x)$, then
$$\lambda_{\min}(a_1,l_-,l_+)\le \lambda_{\min}(a_2,l_-,l_+),\quad
\lambda_{\max}(a_1,l_-,l_+)\le \lambda_{\max}(a_2,l_-,l_+).
$$

\item[(3)] If $l_{2,-}\le l_{1,-}$ and $l_{1,+}\le l_{2,+}$, then
$$\lambda_{\min}(a,l_{1,-},l_{1,+})\le \lambda_{\min}(a,l_{2,-},l_{2,+}),\quad
\lambda_{\max}(a,l_{1,-},l_{1,+})\le \lambda_{\max}(a,l_{2,-},l_{2,+}).
$$

\item[(4)] $\limsup_{l_+-l_-\to 0} \lambda_{\max}(a,l_-,l_+)=-\infty$ and $\liminf_{l_+-l_-\to\infty}\lambda_{\min}(a,l_-,l_+)\ge a_{\inf}(>0)$.

\end{itemize}
\end{lem}

By Lemma \ref{spectrum-lm4}, there is a unique $l^{**}>0$ such that
\begin{equation}
\label{l-star-star-eq}
\begin{cases}
\lambda_{\min}(a,l_-,l_+)>0\quad \forall\,\, l_-<l_+\,\, {\rm with}\,\,  l_+-l_->l^{**}\cr
 \lambda_{\min}(a,l_-,l_+)=0\,\, \text{for some}\,\, l_-<l_+\,\, {\rm with}\,\, l_+-l_-=l^{**}.
\end{cases}
\end{equation}

\subsection{Fisher-KPP equations on bounded fixed domains}

In this subsection, we present some results on the asymptotic behavior of positive solutions of Fisher-KPP equations on bounded fixed domains.
First, consider
\begin{equation}
\label{fisher-kpp-eq1}
\begin{cases}
u_t=u_{xx}+\beta(t,x) u_x +u(a(t,x)-b(t,x) u),\quad 0<x<l\cr
u_x(t,0)=u(t,l)=0,
\end{cases}
\end{equation}
where $l$ is  positive constant and $\beta(t,x)$ is a bounded $C^1$ function. For given $u_0\in X(l)$ with $u_0\ge 0$, let $u(t,x;u_0,\beta,a,b)$ be the solution of
\eqref{fisher-kpp-eq1} with $u(0,x;u_0,\beta,a,b)=u_0(x)$.

\begin{lem}
\label{fisher-kpp-lm1}
If $\lambda(a_{\inf},l)>0$, then there is $\beta_0>0$ such that for any $\beta(t,x)$ with $|\beta(t,x)|\le\beta_0$,
 there is a unique  positive entire  solution $u^*(t,\cdot;\beta,a,b)\in X^{++}(l)$ of \eqref{fisher-kpp-eq1} such that
for any $u_0\in X^+(l)\setminus\{0\}$,
$$
\lim_{t\to\infty} \|u(t,\cdot;u_0,\beta,a,b)-u^*(t,\cdot;\beta,a,b)\|_{X(l)}=0.
$$
\end{lem}

\begin{proof}
Consider
\begin{equation}
\label{linear-eq3}
\begin{cases}
u_t=u_{xx}+\beta(t,x)u_x+a(t,x) u,\quad 0<x<l\cr
u_x(t,0)=u(t,l)=0.
\end{cases}
\end{equation}
Similarly, we can define the principal spectrum interval
$[\lambda_{\min}(\beta,a,l),\lambda_{\max}(\beta,a,l)]$ of \eqref{linear-eq3}.
By $\lambda(a,l)>0$ and \cite[Theorem 4.4.3]{MiSh2}, there is $\beta_0>0$ such that for any $\beta(t,x)$ with
$|\beta(t,x)|\le \beta_0$,
$$
\lambda_{\min}(\beta,a,l)>0.
$$
The lemma then follows from the arguments of \cite[Corollary 3.4]{HeSh}.
\end{proof}

Next, consider
\begin{equation}
\label{fisher-kpp-eq2}
\begin{cases}
u_t=u_{xx}+\beta(t,x) u_x +u(a_0-b_0 u),\quad l_1<x<l_2\cr
u(t,l_1)=u(t,l_2)=0,
\end{cases}
\end{equation}
where $a_0, b_0$ are positive constants, $l_1<l_2$, and $\beta(t,x)$ is a bounded continuous function.

\begin{lem}
\label{fisher-kpp-lm2}
If $\lambda(a_0,l_1,l_2)>0$, then there is $\beta_0>0$ such that for any $\beta(t,x)$ with
$|\beta(t,x)|\le \beta_0$, \eqref{fisher-kpp-eq2} has a unique positive entire solution
$u^*(t,x;\beta,a_0,b_0)$ such that for any $u_0\in C^1([l_1,l_2])\setminus\{0\}$ with $u_0(l_1)=u_0(l_2)=0$ and
$u_0(x)\ge 0$ for $x\in (l_1,l_2)$,
$$
\lim_{t\to\infty} \|u(t,\cdot;u_0,\beta,a_0,b_0)-u^*(t,\cdot;\beta,a_0,b_0)\|_{C([l_1,l_2])}=0,
$$
where $u(t,x;u_0,\beta,a_0,b_0)$ is the solution of \eqref{fisher-kpp-eq2} with
$u(0,x;u_0,\beta,a_0,b_0)=u_0(x)$.
\end{lem}

\begin{proof}
It can be proved by the similar arguments as those in Lemma \ref{fisher-kpp-lm1}.
\end{proof}

\section{Existence and uniqueness of globally defined solutions of the free boundary problems}

In this section, we study the existence and uniqueness of globally defined solutions of \eqref{one-free-boundary-eq} with nonnegative initial functions and prove Theorem \ref{free-boundary-thm1}. To do so, we first prove
two lemmas.

{ For given $T>0$ and a continuous function $h:[0,T]\to (0,\infty)$, let
$$
D_{T,h(\cdot)}=\{(t,x)\in\RR^2\,|\, x\in [0,h(t)],\,\, t\in [0,T]\}.
$$
For given $\alpha,\beta\in [0,1)$ and $u:D_{T,h(\cdot)}\to \RR$, let
$$
|u|_{\alpha,\beta}=\sup_{(t,x),(s,y)\in D_{T,h(\cdot)}, (t,x)\not =(s,y)}\frac{|u(t,x)-u(s,y)|}{|t-s|^\alpha+|x-y|^\beta}.
$$
For given $(k,l)\in \mathbb{Z}^+\times \mathbb{Z}^+$ and $\alpha,\beta\in [0,1)$,
let
\begin{align*}
C^{k+\alpha,l+\beta}(D_{T,h(\cdot)})=\{u(\cdot,\cdot)\in C(D_{T,h(\cdot)})\,| &\, \frac{\p ^i u}{\p t^i}\in C(D_{T,h(\cdot)}) \, (1\le i\le k),\,\,
\frac{\p ^j u}{\p x^j}\in C(D_{T,h(\cdot)})\\
& (1\le j\le l),\,|\frac{\p ^k u}{\p t^k}|_{\alpha,\beta}<\infty,\,\, |\frac{\p ^l u}{\p x^l}|_{\alpha,\beta}<\infty\}
\end{align*}
equipped with the norm
$$
\|u\|_{C^{k+\alpha,l+\beta}(D_{T,h(\cdot)})}=\|u\|_{C(D_{T,h(\cdot)})}+\sum_{i=1}^k \|\frac{\p^i u}{\p t^i}\|_{C(D_{T,h(\cdot)})}
+\sum_{j=1}^l \|\frac{\p ^j u}{\p x^j}\|_{C(D_{T,h(\cdot)})}+|\frac{\p ^k u}{\p t^k}|_{\alpha,\beta}
+|\frac{\p ^l u}{\p x^l}|_{\alpha,\beta}.
$$
}

The first lemma is on the existence and uniqueness of (local) solutions of \eqref{one-free-boundary-eq} with nonnegative initial functions
{ $(u_0,h_0)$ satisfying \eqref{initial-cond-eq}.}

\begin{lem}[Local existence]
\label{local-existence-lm}
For any given $h_0>0$, nonnegative $u_0\in C^2([0,h_0])$ with $u^{'}_0(0)=0$ and $u_0(h_0)=0$, and $\alpha \in (0,1)$, there is $T>0$ such that
the system \eqref{one-free-boundary-eq} admits a unique solution
\begin{eqnarray*}
 &&(u, v_1,v_2,h)\in {C^{\alpha/2,1+\alpha}(D_T)}\times C^{0,\alpha}(D_T)\times C^{0,\alpha}(D_T)\times C^{1+\alpha/2}([0,T])
\end{eqnarray*}
with
$u(0,x)=u_0(x)$ and  $h(0)=h_0$.
Moreover
\begin{equation}
\|u\|_{C^{\alpha/2,1+\alpha}(D_T)} + \|v_1\|_{C^{0,\alpha}(D_T)} + \|v_2\|_{C^{0,\alpha}(D_T)} + \|h\|_{C^{1+\alpha/2}([0,T])} \leq C
\end{equation}
where $D_T = D_{T,h(\cdot)}=\{(t,x) \in \mathbb{R}^2\,|\, x\in [0,h(t)], t\in [0,T]\}$ and $C$  only depends on $h_0,\alpha, T,$ and $\|u_0\|_{C^2([0,h_0])}$
\end{lem}

\begin{proof}
The lemma  can be proved by the similar arguments as those in
\cite[Theorem 2.1]{DuLi}. But, due to the presence of the chemotaxis, nontrivial modifications of the arguments  in
\cite[Theorem 2.1]{DuLi}
are needed.
 For the completeness, we provide a proof in the following.

 As in \cite{Chen2000Afree}, we first straighten the free boundary. Let $\zeta(y)$ be a function in $C^3[0,\infty)$ satisfying
\begin{equation*}
 \zeta(y) = 1 \quad \mbox{if}\quad |y - h_0| < \frac{h_0}{4}, \quad \zeta(y) = 0 \quad \mbox{if} \quad |y-h_0| > \frac{h_0}{2}, \quad |\zeta'(y)| <\frac{6}{h_0}\quad  \forall y\ge 0.
\end{equation*}
We introduce a transformation that will straighten the free boundary:
\begin{equation*}
 (t,y) \rightarrow (t,x), \quad\mbox{where}\quad x = y +\zeta(y)(h(t) - h_0),\quad 0 \leq y<\infty.
\end{equation*}
As long as
\begin{equation*}
 |h(t) - h_0| \leq \frac{h_0}{8},
\end{equation*}
the above transformation is a diffeomorphism from { $[0,+\infty)\times [0,h_0]$ onto $[0,+\infty)\times [0,h(t)]$}. Moreover, it changes the free boundary $x = h(t)$ to the fixed boundary $y = h_0$. One easily computes that
\begin{eqnarray*}
 \frac{\partial y}{\partial x} &=& \frac{1}{1 + \zeta'(y)(h(t) - h_0)} \equiv \sqrt{A(h(t),y)},
 \\
  \frac{\partial^2 y}{\partial x^2} &=& -\frac{1}{[1 + \zeta'(y)(h(t) - h_0)]^3} \equiv B(h(t),y),
  \\
  -\frac{1}{h'(t)} \frac{\partial y}{\partial t} &=& \frac{\zeta(y)}{1 + \zeta'(y)(h(t) - h_0)} \equiv C(h(t),y).
\end{eqnarray*}
Defining
\begin{eqnarray*}
 u(t,x) = u(t,y +\zeta(y)(h(t) - h_0)) = \omega (t,y),
 \\
 v_1(t,x) = v_1(t,y +\zeta(y)(h(t) - h_0)) = z_1 (t,y),
 \\
 v_2(t,x) = v_2(t,y +\zeta(y)(h(t) - h_0)) = z_2 (t,y),
\end{eqnarray*}
then
\begin{equation*}
 u_t = \omega_t - h'(t) C(h(t),y) \omega_y,\quad u_x = \sqrt{A(h(t),y) }\omega_y,
\end{equation*}
\begin{equation*}
 u_{xx} = A(h(t),y) \omega_{yy} + B(h(t),y) \omega_y.
\end{equation*}
Hence the free boundary problem \eqref{one-free-boundary-eq} becomes
\begin{eqnarray}\label{new-FBP2}
\begin{cases}
\omega_t =  A\omega_{yy} + (B + h'C)\omega_y - \chi_1  A \omega_y  z_{1,y} + \chi_2  A \omega_y  z_{2,y}
\\
\qquad - \chi_1 \omega (\lambda_1z_1 -\mu_1\omega) + \chi_2 \omega (\lambda_2z_2 -\mu_2\omega) + \omega (a - b\omega), & y\in (0,h_0)
\\
 0 = Az_{1,yy} + Bz_{1,y} + \mu_1\omega - \lambda_1z_1, &  y\in (0,h_0)
 \\
 0 = Az_{2,yy} + Bz_{2,y} + \mu_2\omega - \lambda_2z_2, & y\in (0,h_0)
 \\
 h'(t) = -\nu \omega_y(t,h_0)
\\
\omega_y (t,0)= z_{1y}(t,0)= z_{2y}(t,0)=0
\\
 \omega(t,h_0)=z_{1y}(t,h_0)= z_{2y}(t,h_0) = 0,
\end{cases}
\end{eqnarray}
where $A = A(h(t),y), B = B(h(t),y),$ and $C = C(h(t),y)$.

Next, denote $h_0^1 = -\nu u_0'(h_0)$, and for $0<T\leq [\frac{h_0}{8(1+h_0^1)}]$, define $\Delta_T = [0,T]\times [0,h_0]$,
\begin{eqnarray*}
 \mathcal{D}_{1T} &=& \{\omega \in C(\Delta_T)\,|\, \omega(0,y) = u_0(y), \|\omega - u_0\|_{C(\Delta_T)} \leq 1\},
 \\
 \mathcal{D}_{2T} &=& \{h \in C^{1}([0,T])\,|\,  h(0) = h_0, h'(0)=h_0^1,  \|h' - h_0^1\|_{C([0,T])
 } \leq 1\}.
\end{eqnarray*}
It is easily seen that $\mathcal{D}:= \mathcal{D}_{1T}\times\mathcal{D}_{2T}$ is a complete metric space with the metric
\begin{equation*}
 d((\omega_1, h_1), (\omega_2, h_2)) = \|\omega_1 - \omega_2\|_{C(\Delta_T)} + \| h_1' - h_2'\|_{C([0,T])}.
\end{equation*}
Let us note that for $h_1,h_2 \in \mathcal{D}_{2T},$ due to $h_1(0) = h_2(0) = h_0$,
\begin{equation}\label{Est-h}
 \| h_1 - h_2\|_{C([0,T])} \leq T\|h_1' - h_2'\|_{C([0,T])}.
\end{equation}
We shall prove the existence and uniqueness result by using the contraction mapping theorem.

 To this end, we  first observe that, due to the choice of $T$, for any given $(\omega,h) \in \mathcal{D}_{1T} \times\mathcal{D}_{2T},$ we have
\begin{equation*}
 |h(t) - h_0|\leq T(1 + h_0^1) \leq \frac{h_0}{8}.
\end{equation*}
Therefore the transformation $(t,y)\rightarrow(t,x)$ introduced at the beginning of the proof is well defined. Applying standard $L^p$ theory (see \cite[Theorem 9.2.5]{Wu2006elliptic}), existence theorem (see local existence result in \cite{Tello2007chemotaxis}, \cite{Tao2013competing}  or semigroup approach developed in \cite{Amann1995linear} (Theorem 5.2.1)) and the Sobolev imbedding theorem \cite{Ladyzenskaja1968linear}, we find that for any $(\omega,h)\in \mathcal{D}$, the following initial boundary value problem
\begin{eqnarray}\label{new-FBP3}
\begin{cases}
\overline{\omega}_t =  A\overline{\omega}_{yy} + (B + h'C)\overline{\omega}_y - \chi_1  A \overline{\omega}_y  \overline{z}_{1,y} + \chi_2  A \overline{\omega}_y  \overline{z}_{2,y}
\\
\quad - \chi_1 \omega (\lambda_1\overline{z}_1 -\mu_1\omega) + \chi_2 \omega (\lambda_2\overline{z}_2 -\mu_2\omega) + \omega (a - b\omega) & y\in (0,h_0)
\\
 0 = A\overline{z}_{1,yy} + B\overline{z}_{1,y} + \mu_1\omega - \lambda_1\overline{z}_1, & y\in (0,h_0)
 \\
 0 = A\overline{z}_{2,yy} + B\overline{z}_{2,y} + \mu_2\omega - \lambda_2\overline{z}_2, & y\in (0,h_0)
\\
\overline{\omega}_y(t,0) = \overline{z}_{1y}(t,0)= \overline{z}_{2y}(t,0) =0
\\
 \overline{\omega}(t,h_0)=\overline{z}_{1y}(t,h_0)= \overline{z}_{2y}(t,h_0)= 0
 \\
\overline{\omega}(y,0) = u_0(y),& 0\leq y\leq h_0,
\end{cases}
\end{eqnarray}
admits a unique solution $(\overline{\omega}, \overline{z}_1,\overline{z}_2)\in { C^{\alpha/2,1+\alpha}(\Delta_T)}\times C^{0,\alpha}(\Delta_T)\times C^{0,\alpha}(\Delta_T)$, and
\begin{equation}\label{Est-Omega}
 \|\overline{\omega}\|_{{ C^{\alpha/2,1+\alpha}(\Delta_T)}}\leq C_1,
\end{equation}
where $C_1$ is a constant depending  on $h_0, \alpha,$ and $\|u_0\|_{C^2[0,h]}$.

Defining
\begin{equation}\label{Eq-h}
 \overline{h}(t) = h_0 - \int_0^t \nu \overline{\omega}_y(\tau,h_0)d\tau,
\end{equation}
we have
\begin{equation*}
 \overline{h}'(t) = -\nu\overline{\omega}_y(t,h_0),\quad \overline{h}(0) = h_0,\quad \overline{h}'(0) = -\nu\overline{\omega}_y(0,h_0) = h_0^1,
\end{equation*}
and hence $\overline{h}'\in C^{\alpha/2}([0,T])$ with
\begin{equation}\label{Est-h}
 \|\overline{h}'\|_{C^{\alpha/2}([0,T])} \leq C_2 := \nu C_1.
\end{equation}
Define $\mathcal{F}: \mathcal{D} \rightarrow C(\Delta_T)\times C^1([0,T])$ by
\begin{equation*}
 \mathcal{F}(\omega,h) = (\overline{\omega},\overline{h}).
\end{equation*}
Clearly $(\omega,h)\in \mathcal{D}$ is a fixed point of $\mathcal{F}$ if and only if it solves \eqref{new-FBP3}.

By \eqref{Est-Omega} and \eqref{Est-h}, we have
\begin{eqnarray*}
 \|\overline{h}' - h_0^1\|_{C([0,T])} \leq \|\overline{h}'\|_{C^{\alpha/2}([0,T])} T^{\alpha/2} \leq \nu C_1T^{\alpha/2},
 \\
 \|\overline{\omega} - u_0\|_{C(\Delta_T)} \leq \|\overline{\omega} - u_0\|_{{ C^{\alpha/2,0}(\Delta_T)}}T^{\alpha/2} \leq C_1 T^{\alpha/2}.
\end{eqnarray*}
Therefor if we take $T\leq\min\{(\nu C_1)^{-2/\alpha}, { C_1^{-2/\alpha}}\},$ then $\mathcal{F}$ maps $\mathcal{D}$ into itself.

Next we prove that $\mathcal{F}$ is a contraction mapping on $\mathcal{D}$ for $T>0$ sufficiently small. Let $(\omega_i,h_i) \in \mathcal{D} (i=1,2)$ and denote $(\overline{\omega}_i,\overline{h}_i) = \mathcal{F}(\omega_i,h_i)$. Then it follows from \eqref{Est-Omega} and \eqref{Est-h} that
\begin{equation*}
 \|\overline{\omega}_i\|_{{ C^{\alpha/2, 1+\alpha}(\Delta_T)}} \leq C_1,\quad \|\overline{h}_i'(t)\|_{C^{\alpha/2}([0,T])} \leq C_2.
\end{equation*}
Choosing two different solutions $(\overline{\omega}_1,\overline{z}_1^1,\overline{z}_2^1,\omega_1,h_1),(\overline{\omega}_2,\overline{z}_1^2,\overline{z}_2^2,\omega_2,h_2)$ and setting $U=\overline{\omega}_1 -\overline{\omega}_2, Z_1 = \overline{z}^1_1 - \overline{z}^2_1, Z_2 = \overline{z}^1_2 - \overline{z}^2_2$, we find that $U, Z_1,$ and $Z_2$ satisfy
\begin{align*}
\begin{cases}
&U_t - A(h_2,y)U_{yy} - [B(h_2,y) + h_2'(t)C(h_2,y)]U_y
\\
&= [A(h_1,y)-A(h_2,y)]\overline{\omega}_{1,yy} + [B(h_1,y)- B(h_2,y) + h_1'C(h_1,y) - h_2'C(h_2,y)]\overline{\omega}_{1,y}
\\
&\quad  + [a -\chi_1\lambda_1 \overline{z}^2_1 +\chi_2\lambda_2\overline{z}^2_2 -(b-\chi_1\mu_1+\chi_2\mu_2)(\omega_1 +\omega_2)](\omega_1-\omega_2)
\\
&\quad  +[\chi_1\lambda_1(\overline{z}^2_1-\overline{z}^1_1) + \chi_2\lambda_2(\overline{z}^1_2-\overline{z}^2_2]\omega_1,\qquad\qquad\qquad\qquad \quad y\in (0,h_0)
\\
& 0= A(h_1,y)Z_{1,yy} + (A(h_1,y)-A(h_2,y)\overline{z}^2_{1,yy} + B(h_1,y)Z_{1,y} + [B(h_1,y)-B(h_2,y)]\overline{z}^2_{1,y}
\\
&\quad  + \mu_1(\omega_1 -\omega_2) -\lambda_1Z_1, \qquad\qquad \qquad\qquad\qquad\qquad\qquad \quad y\in (0,h_0)
\\
&0= A(h_1,y)Z_{2,yy} + (A(h_1,y)-A(h_2,y)\overline{z}^2_{2,yy} + B(h_1,y)Z_{2,y} + [B(h_1,y)-B(h_2,y)]\overline{z}^2_{2,y}
\\
&\quad  + \mu_2(\omega_1 -\omega_2) -\lambda_2Z_2, \qquad\qquad \qquad\qquad\qquad\qquad \qquad\quad y\in (0,h_0)
\\
& U_y(t,0) = Z_{1y}(t,0)= Z_{2y}(t,0) =0
\\
& U(t,h_0)=Z_{1y}(t,h_0)= Z_{2y}(t,h_0)= 0
 \\
&U(0,y) = 0,\quad  0\leq y\leq h_0.
\end{cases}
\end{align*}
Using the $L^p$ estimates for elliptic and parabolic equations and Sobolev's imbedding theorem again, we obtain
\begin{eqnarray}
\|Z_1\|_{C^{0, \alpha}(\Delta_T)} &\leq& C_3^1(\|\omega_1 - \omega_2\|_{C(\Delta_T)} + \|h_1 - h_2\|_{C^{1}([0,T])})
\\
\|Z_2\|_{C^{0, \alpha}(\Delta_T)} &\leq& C_3^2(\|\omega_1 - \omega_2\|_{C(\Delta_T)} + \|h_1 - h_2\|_{C^{1}([0,T])})
\\
\|\overline{\omega}_1 -\overline{\omega}_2\|_{{ C^{\alpha/2, 1+\alpha}(\Delta_T)}} &\leq& C_3^3(\|\omega_1 - \omega_2\|_{C(\Delta_T)} + \|h_1 - h_2\|_{C^1([0,T])}),\label{Est-Omega2}
\end{eqnarray}
where $C_3^i, i=1,2,3$ depends on $C_1, C_2$,  and the functions $A,B$, and $C$ are in the definitions of the transformation $(t,y)\rightarrow (t,x)$. Taking the difference of the equation for $\overline{h}_1, \overline{h}_2$ results in
\begin{equation}\label{Est-h2}
 \|\overline{h}_1' - \overline{h}_2'\|_{C^{\alpha/2}([0,T])}\leq \nu (\|\overline{\omega}_{1y} - \overline{\omega}_{2y}\|_{C^{\alpha/2,0}(\Delta_T)}).
\end{equation}
Combining \eqref{Est-h}, \eqref{Est-Omega2} and \eqref{Est-h2}, and assuming $T\leq 1$, we obtain
\begin{eqnarray*}
 \|\overline{\omega}_1 -\overline{\omega}_2\|_{{ C^{\alpha/2, 1+\alpha}(\Delta_T)}}
 + \|\overline{h}_1' -\overline{h}_2'\|_{C^{\alpha/2}([0,T])}
 \\
 \leq C_4 (\|\omega_1 -\omega_2\|_{C(\Delta_T)} + \|{ h'_1 - h'_2}\|_{C[0,T]})
\end{eqnarray*}
with $C_4$ depending on $C_3$ and $\nu$. Hence for
\begin{equation*}
 T :=\min\{1,(\frac{1}{2C_4})^{2/\alpha}, (\nu C_1)^{-2/\alpha}, { C_1^{-2/\alpha}}, \frac{h_0}{8(1+h_1)}\}
\end{equation*}
we have
\begin{eqnarray*}
 &&\|\overline{\omega}_1 - \overline{\omega}_2\|_{C(\Delta)} + \|\overline{h}_1' - \overline{h}_2'\|_{C([0,T])}
 \\
 &&\leq { T^{\alpha/2}}\|\overline{\omega}_1 -\overline{\omega}_2\|_{{ C^{\alpha/2, 1+\alpha}(\Delta_T)}} + T^{\alpha/2}\|\overline{h}_1' - \overline{h}_2'\|_{C^{\alpha/2}([0,T])}
 \\
 &&\leq C_4 T^{\alpha/2} (\|\omega_1 - \omega_2\|_{C(\Delta_T)} + \| h_1' - h_2'\|_{C([0,T])})
 \\
 &&\leq \frac{1}{2} (\|\omega_1 - \omega_2\|_{C(\Delta_T)} + \| h_1' - h_2'\|_{C([0,T])}).
\end{eqnarray*}
The above shows that $\mathcal{F}$ is a contraction mapping on $\mathcal{D}$ for this $T$. It follows from the contraction mapping theorem that $\mathcal{F}$ has a unique fixed point $(\omega,h)$ in $\mathcal{D}$. In other word, $(\omega(t,y), z_1(t,y), z_2(t,y),h(t))$ is a unique local solution of the problem \eqref{new-FBP2}.
\end{proof}

Let $(u(t,x;u_0,h_0),v_1(t,x;u_0,h_0), v_2(t,x;u_0,h_0),h(t;u_0,h_0))$ be the solution of \eqref{one-free-boundary-eq}
with $u(0,\cdot;u_0,h_0)=u_0(\cdot)\in C^2([0,h_0])$ and $h(0;u_0, h_0)=h_0(>0)$ for $t\in [0,T]$.
The second lemma is on the estimates of $v_1$ and $v_2$.

\begin{lem}
\label{apriori-estimate-lm1}
Assume {\bf (H1)} holds. Suppose that $(u(t,x),v_1(t,x),v_2(t,x),h(t))$ is a nonnegative solution of \eqref{one-free-boundary-eq}
on $[0,T]$ with $u(0,\cdot)=u_0(\cdot)$, which satisfies \eqref{initial-cond-eq}. Then
$$
(\chi_2 \lambda_2 v_2-\chi_1 \lambda_1 v_1)(x,t;u_0,h_0)\leq M \|u(t,\cdot)\|_\infty\quad \forall\, t\in [0,T],
$$
where $M$ is as in \eqref{m-eq}.
In particular, if $\|u(t,\cdot)\|_\infty\le \max\{\|u_0\|_\infty,M_0\}$ for $0\le t\le T$, then
$$
(\chi_2 \lambda_2 v_2-\chi_1 \lambda_1 v_1)(x,t;u_0,h_0)\leq M \max\{\|u_0\|_\infty, M_0\}\quad \forall\, t\in [0,T],
$$
where $M_0$ is as in \eqref{M0-eq}.
\end{lem}

\begin{proof} It can be proved by the similar arguments as those in \cite[Lemma 2.2]{BaoShen1}. For completeness, we provide a proof in the
following.

Note that $v_1(t,x)$ is the solution of
$$
\begin{cases}
v_{1,xx}-\lambda_1 v_1+\mu_1 u(t,x)=0,\quad 0<x<h(t)\cr
v_{1x}(t,0)=v_{1x}(t,h(t))=0
\end{cases}
$$
and $v_2(t,x)$ is the solution of
$$
\begin{cases}
v_{2,xx}-\lambda_2 v_2+\mu_2 u(t,x)=0,\quad 0<x<h(t)\cr
v_{2x}(t,0)=v_{2x}(t,h(t))=0.
\end{cases}
$$
Let $T(s)$ be the semigroup generated by $\partial_{xx}$ on $(0,h(t))$ with Neumann boundary condition.
Then
$$
v_i(t,\cdot)=\mu_i \int_0^ \infty e^{-\lambda_i s} T(s) u(t,\cdot)ds,\quad i=1,2.
$$
We then have
\begin{align*}
(\chi_2 \lambda_2 v_2-\chi_1 \lambda_1 v_1)(x,t;u_0)&=\chi_2 \lambda_2 \mu_2 \int_0^ \infty e^{-\lambda_2 s} T(s) u(t,\cdot)ds-
\chi_1 \lambda_1 \mu_1 \int_0^ \infty e^{-\lambda_1 s} T(s) u(t,\cdot)ds\\
&=\big(\chi_2 \lambda_2 \mu_2-\chi_1\lambda_1\mu_1\big) \int_0^ \infty e^{-\lambda_2 s} T(s) u(t,\cdot)ds\\
&\,\, \,\,  +
\chi_1 \lambda_1 \mu_1 \int_0^ \infty \big(e^{-\lambda_2 s}-e^{-\lambda_1 s}\big) T(s) u(t,\cdot)ds\\
& \le \big(\chi_2 \lambda_2 \mu_2-\chi_1\lambda_1\mu_1\big)_+ \int_0^ \infty e^{-\lambda_2 s} T(s) u(t,\cdot)ds\\
&\,\, \,\,  +
\chi_1 \lambda_1 \mu_1 \int_0^ \infty \big(e^{-\lambda_2 s}-e^{-\lambda_1 s}\big)_+ T(s) u(t,\cdot)ds\\
\end{align*}
Note that
$$
T(s)u(t,\cdot)\le T(s) \|u(t,\cdot)\|=\|u(t,\cdot)\|_\infty.
$$
Hence
\begin{align*}
(\chi_2 \lambda_2 v_2-\chi_1 \lambda_1 v_1)(x,t;u_0,h_0)&\le  \big(\chi_2 \lambda_2 \mu_2-\chi_1\lambda_1\mu_1\big)_+ \int_0^ \infty e^{-\lambda_2 s} \|u(t,\cdot)\|_\infty ds\\
&\,\,\,\,  +
\chi_1 \lambda_1 \mu_1 \int_0^ \infty \big(e^{-\lambda_2 s}-e^{-\lambda_1 s}\big)_+ \|u(t,\cdot)\|_\infty ds\\
&= \frac{\|u(t,\cdot)\|_\infty}{\lambda_2}\Big((\chi_2\lambda_2\mu_2-\chi_1\lambda_1\mu_1)_{+}+\chi_1\mu_1(\lambda_1-\lambda_2)_{+}\Big).
\end{align*}

Similarly, we can prove that
$$
 (\chi_2 \lambda_2 v_2-\chi_1 \lambda_1 v_1)(x,t;u_0,h_0)\le \frac{\|u(t,\cdot)\|_\infty}{\lambda_{1}}\Big( \chi_2\mu_2(\lambda_1-\lambda_2)_{+} + (\chi_2\mu_2\lambda_2-\chi_1\mu_1\lambda_1)_{+} \Big).
$$
The lemma then follows.
\end{proof}

The third  lemma is on the estimate of $u(t,x;u_0,h_0)$.

\begin{lem}
\label{apriori-estimate-lm2}
If ({\bf H1}) holds. For given $T>0$, $h_0>0$, and $u_0\in C^2([0,h_0])$ with $u_0(x)\ge 0$ for $x\in [0,h_0]$, $u_0^{'}(0)=0$, and $u_0(h_0)=0$, assume that
\eqref{one-free-boundary-eq} has a unique solution $(u(t,x),v_1(t,x),v_2(t,x),h(t))$ on $t\in [0,T]$ with $u(0,x)=u_0(x)$ and
$h(0)=h_0$. Then
$$
\|u(t,\cdot)\|\infty\le \max\{\|u_0\|_\infty, M_0\}\quad \forall\, \, 0\le t\le T,
$$
where $M_0$ is as in \eqref{M0-eq}. If $\|u(t,\cdot)\|_\infty>M_0$ on $[0,T]$, then
$\|u(t,\cdot)\|_\infty$ is decreasing in $t$.
\end{lem}

\begin{proof}
Let $C_0=\max\{\|u_0\|_\infty,M_0\}$ and
$$\mathcal{E}:=\{u\in  C_{\rm unif}^b([0,h(t)]\times[0, T])\,|\, u(\cdot,0)=u_0, 0\leq u(x,t)\leq C_0, x\in [0,h(t)], 0\leq t\leq T\}
  $$
  equipped with the norm
$$
\|u(\cdot,\cdot)\|=\sup_{0\le t\le T,0\le x\le h(t)} |u(t,x)|.
$$
It is clear that $\mathcal{E}$ is closed and convex.

For any given $u\in \mathcal{E}$, let $V_i(t,x;u)$ $(i=1,2)$  be the solution of
$$
\begin{cases}
0=V_{i,xx}-\lambda_i V_i+\mu_i u(t,x),\quad 0<x<h(t)\cr
V_{i,x}(t,0)=V_{i,x}(t,h(t))=0.
\end{cases}
$$
Let $U(t,x;u)$ be the solution of
$$
\begin{cases}
U_t=U_{xx}-\chi_1 U_x V_{1,x}+\chi_2 U_x V_{2,x}\cr
\qquad + U\big(a-(\chi_1\lambda_1 V_1-\chi_2 \lambda_2 V_2)-
(b-\chi_1 \mu_1+\chi_2\mu_2) U\big),\,\, 0<x<h(t)\cr
U_x(t,0)=U(t,h(t))=0\cr
U(0,x)=u_0(x),\quad 0\le x\le h(0).
\end{cases}
$$
By the arguments of Lemma \ref{apriori-estimate-lm1},
$$
U_t\le U_{xx}-\chi_1 U_x V_{1,x}+\chi_2 U_x V_{2,x} + U\big(a+M C_0-
(b-\chi_1 \mu_1+\chi_2\mu_2) U\big),\,\, 0<x<h(t).
$$
Note that
$$
a+MC_0 -(b-\chi_1\mu_1+\chi_2\mu_2)C_0\le 0.
$$
Then by comparison principle for parabolic equations,
$$
U(t,x;u)\le C_0\quad \forall \, 0\le t\le T,\,\, 0\le x\le h(t).
$$
Hence $U(\cdot,\cdot;u)\in \mathcal{E}$.

By the similar arguments as those in  \cite[Lemma 4.3]{Salako2016spreading}, we can prove that the mapping
$\mathcal{E}\ni u\to U(\cdot,\cdot;u)\in\mathcal{E}$ is continuous and compact.
Then there is $u^*\in\mathcal{E}$ such that
$U(t,x;u^*)=u^*(t,x)$. This implies that
$u(t,x)=u^*(t,x)$ for $0\le x\le h(t)$ and $0\le t<T_{\max}(u_0,h_0)$, and then
$$
u(t,x)\le C_0\quad \forall\,\, 0\le t\le T,\,\, 0\le x\le h(t).
$$

Now if $\|u(t,\cdot)\|_\infty>M_0$ for $t\in [0,T]$, then for any $0\le t_1\le t_2\le T$,
$$
\|u(t_2,\cdot)\|_\infty\le \max\{\|u(t_1,\cdot)\|_\infty,M_0\}=\|u(t_1,\cdot)\|_\infty.
$$
The lemma is thus proved.
\end{proof}

We now prove Theorem \ref{free-boundary-thm1}. Without loss of generality, we prove the case that $t_0=0$,
and we denote by $(u(t,x;u_0,h_0),v_1(t,x;u_0,h_0),v_2(t,x;u_0,h_0),h(t;u_0,h_0))$ the solution of \eqref{one-free-boundary-eq}
with $u(0,x;u_0,h_0)=u_0(x)$ and $h(0;u_0,h_0)=h_0(x)$.

\begin{proof}[Proof of  Theorem \ref{free-boundary-thm1}]
Suppose that $[0,T_{\max}(u_0,h_0))$ is the maximal interval of existence of the solution $(u(t,x;u_0,h_0),v_1(t,x;u_0,h_0),v_2(t,x;u_0,h_0),h(t;u_0,h_0))$. Let $h(t)=h(t;u_0,h_0)$, and
$(u(t,x),v_1(t,x),v_2(t,x))=(u(t,x;u_0,h_0),v_1(t,x;u_0,h_0),v_2(t,x;u_0,h_0))$. Then $(u,v_1,v_2)$ satisfies
\begin{equation}
\label{gl-eq2}
\begin{cases}
u_t=u_{xx}-\chi_1 u_x v_{1,x}+\chi_2 u_x v_{2,x}\cr
\qquad + u\big(a-(\chi_1\lambda_1 v_1-\chi_2 \lambda_2 v_2)-
(b-\chi_1 \mu_1+\chi_2\mu_2) u\big),\,\, 0<x<h(t)\cr
0=v_{1,xx}-\lambda_1 v_1+\mu_1 u,\quad 0<x<h(t)\cr
0=v_{2,xx}-\lambda_2 v_2 +\mu_2 u,\quad 0<x<h(t)\cr
u_x(t,0)=u(t,h(t))=v_{1,x}(t,0)=v_{1,x}(t,h(t))=v_{2,x}(t,0)=v_{2,x}(t,h(t))=0\cr
u(0,x)=u_0(x),\quad 0\le x\le h(0).
\end{cases}
\end{equation}

First, by Lemma \ref{apriori-estimate-lm2},
\begin{equation}
\label{aux-new-eq1}
u(t,x)\le C_0\quad \forall \, t\in [0,T_{\max}),\,\, 0\le x\le h(t),
\end{equation}
where $C_0=\max\{\|u_0\|_\infty,M_0\}$.

Next, we show that there is $\tilde M>0$ such that
\begin{equation}
\label{aux-new-eq2}
h^{'}(t)\le \tilde M\quad \forall\, \, t\in [0, T_{\max}).
\end{equation}
In order to do so, for given $M_1>0$,  define
\begin{equation*}
 \Omega = \Omega_{M_1} :=\{(t,x)\,|\, 0<t<T, h(t) - M_1^{-1} < x < h(t)\}
\end{equation*}
and construct an auxiliary function
\begin{equation}\label{Eq-aux1}
 \omega(t,x) = C_0[2M_1(h(t) - x) - M_1^2(h(t) - x)^2].
\end{equation}

From \eqref{Eq-aux1}, we obtain that for $(t,x)\in \Omega$,
\begin{eqnarray*}
 \omega_t &=& 2C_0M_1h'(t)(1-M_1(h(t)-x)) \geq 0,
 \\
 \omega_x &=& -2C_0M_1 + 2M_1^2(h(t)-x),
 \\
 -\omega_{xx} &=& 2C_0M_1^2\\
  u(a-bu) &\leq & aC_0.
\end{eqnarray*}
Compared to \eqref{one-free-boundary-eq}, it follows that
\begin{eqnarray}
\omega_t &-& \omega_{xx} + \chi_1 \omega_x v_{1,x} - \chi_2\omega_x v_{2,x} - u(a - (\chi_1\lambda_1 v_1 - \chi_2\lambda_2 v_2)- \chi_2\mu_2 u +\chi_1\mu_1u-bu)\nonumber
\\
&\geq& 2C_0M_1^2 + (\chi_1v_{1,x}  - \chi_2v_{2,x})[-2C_0M_1 + 2M_1^2(h(t)-x)]\nonumber
 \\
&&- u(a  - (\chi_1\lambda_1 v_1 - \chi_2\lambda_2 v_2) - \chi_2\mu_2 u + \chi_1\mu_1u),\label{Eq-comparison-boundary1}
\end{eqnarray}
with $|h(t)-x| \leq M_1^{-1},\|v_{i,x}\|_{L^\infty(\Omega)} \leq \tilde K \|u\|_{L^\infty([0,T]\times\Omega)} \le \tilde  KC_0, (i=1,2)$  for some constant
$\tilde K>0$. With big enough $M_1$ and Lemma \ref{apriori-estimate-lm1} we can obtain that \eqref{Eq-comparison-boundary1} is positive and
$w(0,x)\ge u_0(x)$ for $h_0-\frac{1}{M_1}\le x\le h_0$.
Note that $w(t,h(t))=u(t,h(t))=0$ and $w(t,h(t)-\frac{1}{M_1})=C_0\ge u(t,h(t)-\frac{1}{M_1})$.  We then can apply the maximum principle to $\omega -u$ over $\Omega$ to deduce that $u(t,x) \leq \omega(t,x)$ for $(t,x)\in \Omega$. It then would follow that
\begin{equation*}
 u_x(t,h(t)) \geq \omega_x(t,h(t)) = -2M_1C_0,\quad h'(t) = -\nu u_x(t,h(t)) \leq \tilde M:= \nu 2M_1C_0,
\end{equation*}
which is independent of $T$, which implies  \eqref{aux-new-eq2}.

We show now that $T_{\max} = \infty$. Assume that $T_{\max} <\infty$.
By  the arguments of Lemma \ref{local-existence-lm}, \eqref{aux-new-eq1}, and
 \eqref{aux-new-eq2}, there exists a $\tau >0$ depending only on $C_0$ and $\tilde M$  such that the solution of the problem \eqref{one-free-boundary-eq} with initial time $T_{\max}-\tau/2, h(T_{\max}-\tau/2) = \frac{19 h_0}{16}$ can be extended uniquely to the time $T_{\max} + \tau/2$ and $h(T_{\max} + \tau/4) > \frac{5}{4}h_0$ with the bounded solutions $(u,v_1,v_2)$, which is a contradiction.
  Hence $T_{\max} =\infty$.

Finally we prove \eqref{thm1-eq1} and \eqref{thm1-eq2} hold.
   By Lemma \ref{apriori-estimate-lm2},  \eqref{thm1-eq1} holds.
   \eqref{thm1-eq2} can be proved by the similar arguments as those of \cite[Lemma 2.3]{BaoShen1}.
  The theorem is thus proved.
\end{proof}

\section{Spreading and vanishing dichotomy}
In this section, we study the spreading and vanishing scenarios for  the chemotaxis free boundary problem \eqref{one-free-boundary-eq}, and prove
Theorem \ref{free-boundary-thm2}. To do so, we first prove two lemmas.

\begin{lem}
\label{thm2-lm1}
Assume that {\bf (H1)} holds.
 Let $h_0>0$ and $u_0(\cdot)$ satisfy \eqref{initial-cond-eq}, and $(u(\cdot,\cdot;u_0,h_0)$, $v_{1}(\cdot,\cdot;u_0,h_0)$, $v_2(\cdot,\cdot;u_0,h_0)$, $h(t;u_0,h_0))$ be the classical solution of \eqref{one-free-boundary-eq} with $u(\cdot,0;u_0,h_0)=u_0(\cdot)$
 and $h(0;u_0,h_0)=h_0$. Then  we have that
\begin{align*}\label{asym-eq16}
&\|\partial_{x}(\chi_2v_2-\chi_1v_1)(\cdot,t;u_0)\|_{\infty}\nonumber \\
& \leq \min\Big\{\frac{|\chi_2\mu_2-\chi_1\mu_1|}{2\sqrt{\lambda_2}}+\frac{\chi_1\mu_1|\sqrt{\lambda_1}-\sqrt{\lambda_2}|}{2\sqrt{\lambda_1\lambda_2}}, \frac{|\chi_1\mu_1-\chi_2\mu_2|}{2\sqrt{\lambda_1}}+\frac{\chi_2\mu_2|\sqrt{\lambda_2}-\sqrt{\lambda_1}|}{2\sqrt{\lambda_1\lambda_2}}
\Big\}\|u(t,\cdot,;u_0,h_0)\|_{\infty}
\end{align*}
for every   $t\geq 0$.
\end{lem}

\begin{proof}
First, define $\tilde u^+(t,x)$ for $x\ge 0$ successfully by
$$
\tilde u^+(t,x)=u(t,x;u_0,h_0)\quad {\rm for}\quad 0\le x\le h(t;u_0,h_0),
$$
$$
\tilde u^+(t,x)=\tilde u^+(t,2h(t;u_0,h_0)-x)\quad {\rm for}\quad h(t;u_0,h_0)<x\le 2h(t;u_0,h_0),
$$
and
$$
\tilde u^+(t,x)=\tilde u^+(t,2k h(t;u_0,h_0)-x)\quad {\rm for}\quad kh(t;u_0,h_0)<x\le (k+1) h(t;u_0,h_0)
$$
for $k=1,2,\cdots$. Define $\tilde u(t,x)$ for $x\in\RR$ by
$$
\tilde u(t,x)=\tilde u^+(t,|x|)\quad {\rm for}\quad x\in\RR.
$$
It is clear that
$$
\tilde u(t,x)=\tilde u(t,-x)\quad \forall\,\, x\in\RR.
$$
and
$$
\tilde u(t,2h(t;u_0,h_0)-x)=\tilde u(t,x)\quad \forall \,\, x\in\RR.
$$

Next, let $\tilde v_i(t,x)$ be the solution of
$$
0=\tilde v_{i,xx}(t,x)-\lambda_i \tilde v_i(t,x)+\mu_i\tilde u(t,x),\quad x\in\RR.
$$
Then
\begin{align*}
0&=\tilde v_{i,xx}(t,-x)-\lambda_i \tilde v_i(t,-x)+\mu_i\tilde u(t,-x)\\
&=\tilde v_{i,xx}(t,-x)-\lambda_i \tilde v_i(t,-x)+\mu_i \tilde u(t,x),\quad x\in\RR
\end{align*}
and
\begin{align*}
0&=\tilde v_{i,xx}(t,2h(t;u_0,h_0)-x)-\lambda_i \tilde v_i(t,2h(t;u_0,h_0)-x)+\mu_i \tilde u(t,2h(t;u_0,h_0)-x)\\
&=\tilde v_{i,xx}(t,2h(t;u_0,h_0)-x)-\lambda_i \tilde v_i(t,2h(t;u_0,h_0)-x)+\mu_i \tilde u(t,x),\quad x\in\RR.
\end{align*}
It then follows that
$$
\tilde v_i(t,x)=\tilde v_i(t,-x)=\tilde v_i(t,2h(t;u_0,x_0)-x)\quad \forall \,\, x\in\RR.
$$
This implies that
$$
\begin{cases}
0=\tilde v_{i,xx}-\lambda_i \tilde v_i+\mu_i u(t,x;u_0,h_0)\quad 0<x<h(t;u_0,h_0)\cr
\tilde v_{i,x}(t,0)=\tilde v_{i,x}(t,h(t;u_0,h_0))=0.
\end{cases}
$$
Therefore,
$$
v_i(t,x;u_0,h_0)=\tilde v_{i}(t,x)\quad \forall\,\, 0\le x\le h(t;u_0,h_0)
$$
and
\begin{equation}
\label{thm2-eq1}
v_i(t,x;u_0,h_0)=\mu_i\frac{1}{2\sqrt \pi}\int_{0}^{\infty}\int_{-\infty}^\infty \frac{e^{-\lambda_0 s}}{\sqrt s}e^{-\frac{|x-z|^{2}}{4s}}\tilde u(t,z)dzds
\end{equation}
for $0\le x\le h(t;u_0,h_0)$ and $i=1,2$.
The lemma then follows from  \cite[Lemma 4.1]{Salako2017Global}.
\end{proof}

\begin{lem}
\label{thm2-lm2}
Assume that {\bf (H1)} holds.
 Let $h_0>0$ and $u_0(\cdot)$ satisfy \eqref{initial-cond-eq}, and $(u(\cdot,\cdot;u_0,h_0)$, $v_{1}(\cdot,\cdot;u_0,h_0)$, $v_2(\cdot,\cdot;u_0,h_0)$, $h(t;u_0,h_0))$ be the classical solution of \eqref{one-free-boundary-eq} with $u(\cdot,0;u_0,h_0)=u_0(\cdot)$
 and $h(0;u_0,h_0)=h_0$.
For every $R\gg 1$, there are $C_R\gg 1$ and $\varepsilon_R>0$ such that for any $t>0$ with $h(t;u_0,h_0)>R$,  we have
 \begin{align}\label{nnew-eq3}
&|\chi_i v_{i,x}(t,\cdot;u_0,h_0)|_{C([0,\frac{R}{2}])} + |\chi_i\lambda_i v_i(t,\cdot;u_0,h_0)|_{C([0,\frac{R}{2}])}\nonumber\\
&\leq C_{R}\|u(t,\cdot;u_0,h_0)\|_{C([0,R])}+\varepsilon_R \max\{\|u_0\|_{C([0,h_0])},M_0\},\quad i=1,2
\end{align}
with $\lim_{R\to\infty}\varepsilon_R=0$,  where $M_0$ is as in \eqref{M0-eq}.
\end{lem}

\begin{proof}
It follows from \eqref{thm2-eq1} and  \cite[Lemma 2.5]{SaShXu}.
\end{proof}

 We now prove Theorem \ref{free-boundary-thm2}.  Without loss of generality, we also prove the case that $t_0=0$,
and we denote by $(u(t,x;u_0,h_0),v_1(t,x;u_0,v_0),v_2(t,x;u_0,h_0),h(t;u_0,h_0))$ the solution of \eqref{one-free-boundary-eq}
with $u(0,x;u_0,h_0)=u_0(x)$ and $h(0;u_0,h_0)=h_0$.

\begin{proof}[Proof of Theorem \ref{free-boundary-thm2}]

  Note that we have either $h_\infty<\infty$ or $h_\infty=\infty$. It then suffices to prove that  $h_\infty<\infty$ and $h_\infty=\infty$
  imply (i) and (ii), respectively.

 First,  suppose that $h_\infty <\infty$. We prove (i) holds.

 To this end, we first claim that $h'(t;u_0,h_0)\rightarrow 0$ as $t\rightarrow\infty$. Assume that the claim is not true. Then there is $t_n\rightarrow\infty$  $(t_n\geq 2)$ such that $\lim_{n\to\infty} h'(t_n;u_0,h_0) >0$. Let $h_n(t) = h(t+t_n;u_0,h_0)$ for $t \geq -1$.
 By the arguments of Theorem \ref{free-boundary-thm1},
 $\{h'_n(t)\}$ is uniformly bounded on $[-1,\infty)$, and then by the arguments of Lemma \ref{local-existence-lm}, $\{h'_n(t)\}$ is equicontinuous on $[-1,\infty)$.
  We may then assume that there is a continuous function $L^*(t)$ such that $h'_n(t)\to L^*(t)$ as $n\to\infty$ uniformly in $t$ in bounded sets of $[-1,\infty).$ It then follows that $L^*(t) =\frac{dh_\infty}{dt} \equiv 0$ and then $\lim_{n\to\infty} h'(t_n;u_0,h_0) =0$, which is a contradiction. Hence the claim holds.

Second, we show that if $h_\infty<\infty$, then  $\lim_{t\to\infty} \|u(t,\cdot;u_0,h_0)\|_{C([0,h(t)])} =0.$
Assume that this is not true. Then by a priori estimates for parabolic equations,  there are $t_n\to \infty$ and $u^*(t,x) \neq 0$ such that $\|u(t+t_n,\cdot;u_0,h_0) - u^*(t,\cdot)\|_{C([0,h(t+t_n)]} \to 0$ as $t_n\to \infty$. Without loss of generality, we may assume that the limits $\lim_{n\to\infty} a(t+t_n,x)$ and $\lim_{n\to\infty} b(t+t_n,x)$
 exist locally uniformly in $(t,x)$. We then  have $u^*(t,\cdot)$ is an entire solution of
\begin{eqnarray*}
\begin{cases}
u_t = u_{xx} -\chi_1  (u  v_{1,x})_x +\chi_2(u v_{2,x})_x+ u(a^*(t,x) - b^*(t,x)u),& \mbox{in}\quad (0,h_\infty)
\\
 0 = v_{1,xx} - \lambda_1v_1 + \mu_1u, & \mbox{in}\quad (0,h_\infty)\\
 0=v_{2,xx}-\lambda_2 v_2+\mu_2 u,  & \mbox{in}\quad (0,h_\infty)
\\
u_x(t,0)=v_{1,x}(t,0)=v_{2,x}(t,0)=0
\\
 u(t,h_\infty)=v_{1,x}(t,h_\infty)=v_{2,x}(t,h_\infty)= 0,
\end{cases}
\end{eqnarray*}
where $a^*(t,x)=\lim_{n\to\infty}a (t+t_n,x)$ and $b^*(t,x)=\lim_{n\to\infty} b(t+t_n,x)$.
By the Hopf Lemma for parabolic equations, we have $u^*_x(t,h_\infty) < 0$. This implies that
\begin{equation*}
 \lim_{n\to\infty} h'(t_n) = - \lim_{n\to\infty} \nu u_x(t_n,h(t_n);u_0,h_0) =-\nu u^*_x(0,h_\infty)>0,
\end{equation*}
which is a contradiction again. Hence $\lim_{t\to\infty} \|u(t,\cdot;u_0,h_0)\|_{C([0,h(t)])} =0.$

 Third, we show  that $h_\infty <\infty$ implies $\lambda(a,h_\infty)\le 0$, which is equivalent to $h_\infty \leq l^*$.
 Assume that $h_\infty \in (l^*,\infty).$ Then, for any $\epsilon>0$,
 there exists $\widetilde{T} >0$ such that $h(t;u_0,h_0) > h_\infty -\epsilon > l^*$ and $\|u(t,\cdot;u_0,h_0)\|_{C([0,h(t;u_0,h_0)])}<\epsilon$
  for all $t \geq \widetilde{T}$.
 Note that $u(t,x;u_0,h_0)$ satisfies
 \begin{equation}
 \label{aux-spreading-eq1}
 \begin{cases}
 u_t = u_{xx} -\chi_1  u_x  v_{1,x}+\chi_2 u_x v_{2,x}\\
 \qquad +  u(a(t,x)-\chi_1\lambda_1 v_1+\chi_2\lambda_2 v_2 -( b(t,x)-\chi_1\mu_1+\chi_1\mu_2)u),& \mbox{in}\quad (0,h_\infty-\epsilon)
\\
u_x(t,0)=0,\,\, u(t,h_\infty-\epsilon)> 0
 \end{cases}
 \end{equation}
 for $t\ge \tilde T$,
 where $v_1(t,x)=v_1(t,x;u_0,h_0)$ and $v_2(t,x)=v_2(t,x;u_0,h_0)$.

Consider
 \begin{eqnarray}\label{Cutoff-Eq}
\begin{cases}
w_t = w_{xx} +\beta(t,x)  w_x +  w(\tilde a(t,x)-\tilde b(t,x)w),& \mbox{in}\quad (0,h_\infty - \epsilon)
\\
w_x(t,0) =0,\,\, w(t, h_\infty -\epsilon)= 0
 \\
 w(\tilde T,x)=\tilde u_0(x),
\end{cases}
\end{eqnarray}
where
$$
\begin{cases}
\beta(t,x)=-\chi_1 v_{1,x}(t,x;u_0,h_0)+\chi_2 v_{2,x}(t,x;u_0,h_0)\cr
\tilde a(t,x)=a(t,x)-\chi_1 \lambda_1 v_1(t,x;u_0,h_0)+\chi_2 \lambda_2 v_2(t,x;u_0,h_0)\cr
\tilde b(t,x)=b(t,x)-\chi_1\mu_1 v_1(t,x;u_0,h_0)+\chi_2\mu_2 v_2(t,x;u_0,h_0),
\end{cases}
$$
and
$$
\tilde u_0(x)=u(\tilde T,x;u_0,h_0).
$$
By comparison principle for the parabolic equations, we have
\begin{equation}
\label{aux-spreading-eq2}
 u(t + \widetilde{T},\cdot;u_0,h_0) \geq w (t + \widetilde{T},\cdot;\tilde T,\tilde u_0)\quad\mbox{for}\quad t\geq 0,
\end{equation}
where $w (t + \widetilde{T},\cdot;\tilde T,\tilde u_0)$ is the solution of  \eqref{Cutoff-Eq} with $w(\widetilde{T},\cdot;\tilde T,\tilde u_0) = \tilde u_0(\cdot)$.

 By Lemma \ref{apriori-estimate-lm1},
 $$
 |\chi_1 \lambda_1 v_1(t,x;u_0,h_0)-\chi_2\lambda_2 v_2(t,x;u_0,h_0)|\le M\epsilon\quad \forall\, t\ge \tilde T,\, \, x\in [0,h(t;u_0,h_0)].
 $$
 Let
 $$
 M_{\chi_1,\chi_2}=\min\Big\{\frac{|\chi_2\mu_2-\chi_1\mu_1|}{2\sqrt{\lambda_2}}+\frac{\chi_1\mu_1|\sqrt{\lambda_1}-\sqrt{\lambda_2}|}{2\sqrt{\lambda_1\lambda_2}}, \frac{|\chi_1\mu_1-\chi_2\mu_2|}{2\sqrt{\lambda_1}}+\frac{\chi_2\mu_2|\sqrt{\lambda_2}-\sqrt{\lambda_1}|}{2\sqrt{\lambda_1\lambda_2}}
\Big\}.
$$
By Lemma \ref{thm2-lm1},
$$
|\chi_1 v_{1,x}(t,x;u_0,h_0)-\chi_2 v_{2,x} (t,x;u_0,h_0)|\le M_{\chi_1,\chi_2} \epsilon \quad \forall\, t\ge \tilde T,\, \, x\in [0,h(t;u_0,h_0)].
 $$
Then by Lemma \ref{fisher-kpp-lm1}, for $0<\epsilon\ll 1$,
$$
\liminf_{t\to\infty} \|w(t,\cdot;\tilde T,\tilde u_0)\|_{C([0,h_\infty-\epsilon])}>0.
$$
This together with \eqref{aux-spreading-eq2} implies that
$$
\liminf_{t\to\infty}\|u(t,\cdot;u_0,h_0)\|_{C([0,h(t;u_0,h_0)])}>0,
$$
which is a contradiction. Therefore, $h_\infty\le l^*$.
 (i) is thus proved.

\medskip

Next, suppose that that $h_\infty =\infty$. We prove that (ii) holds.

To this end, fix $m\gg  l^*$ with
$$2\varepsilon_{2m} \max\{\|u_0\|_{C([0,h_0])},M_0\}<\frac{a_{\inf}}{3},
$$
where $\varepsilon_{2m}$ is as in Lemma \ref{thm2-lm2} with $R=2m$.
 Note that there is $T_0>0$ such that $h(t;u_0,h_0)>2m$ for $t\ge T_0$.
We first prove

\smallskip

\noindent {\bf Claim 1.}  {\it For any $\epsilon>0$, there is $\delta_\epsilon>0$ such that for any $t\ge T_0$,
 if $\sup_{0\le x\le 2m}u(t,x;u_0,h_0)\ge \epsilon$, then $\inf_{0\le x\le m}u(t,x;u_0,h_0)\ge \delta_\epsilon$.
 }

 \smallskip

 In fact, if the claim does not hold, there is $\epsilon_0>0$ and $T_0\le t_n\to \infty$ such that
 $$\sup_{0\le x\le 2m}u(t_n,x;u_0,h_0)\ge \epsilon_0, \quad \inf_{0\le x\le m}u(t_n,x;u_0,h_0)\le \frac{1}{n}.
 $$
 Without loss of generality, we may assume that
 $u(t+t_n,x;u_0,h_0)\to u^*(t,x)$,  $v_i(t+t_n,x;u_0,h_0)\to v^*_i(t,x)$ ($i=1,2$),
 $a(t+t_n,x)\to a^*(t,x)$, and $b(t+t_n,x)\to b^*(t,x)$ as $n\to\infty$.
 Then $u^*(t,x)$ satisfies
 \begin{equation*}
 \begin{cases}
 u_t=u_{xx}-\big((\chi_1 v_1^*(t,x)-\chi_2 v_2^*(t,x) u\big)_x+u(a^*(t,x)-b^*(t,x)u),\quad 0<x<2m\cr
 u_x(t,0)=0,\,\, u(t,2m)>0
 \end{cases}
\end{equation*}
for all $t\in\RR$. Note that $u^*(t,x)\ge 0$ and  $\sup_{0\le x\le 2m}u^*(0,x)\ge \epsilon_0$. Then by comparison principle for
parabolic equations, $u^*(t,x)>0$ for all $t\in\RR$ and $0<x<2m$. In particular,
$\inf_{0\le x\le m}u^*(0,x)>0$, which is a contradiction. Hence
the claim 1 holds.

\smallskip

Next, we prove

\smallskip

\noindent {\bf Claim 2.} {\it  For any $0<\epsilon\ll 1$, there is $\tilde \epsilon>0$ such that
if $\sup_{0\le x\le 2m} u(t,x;u_0,h_0)\le \epsilon$ for $T_0\le  t_1\le t\le t_2$ and
$\sup_{0\le x\le 2m} u(t_1,x;u_0,h_0)=\epsilon$ ($i=1,2$), then
$$
\sup_{0\le x\le 2m} u(t,x;u_0,h_0)\ge \tilde \epsilon\quad \forall\,\, t_1\le t\le t_2.
$$
}

\smallskip

In fact, for given $0<\epsilon \ll 1$, assume $\sup_{0\le x\le 2m} u(t,x;u_0,h_0)\le \epsilon$ for $T_0\le t_1\le t\le t_2$
and $\sup_{0\le x\le 2m} u(t_1,x;u_0,h_0)=\epsilon$. Note that $u(t,x;u_0,h_0)$ satisfies
\begin{equation}
\label{aux-thm2-eq1}
\begin{cases}
u_t=u_{xx}+\tilde \beta(t,x)u_x+u(\tilde a(t,x)-\tilde b(t,x)u),\quad 0<x<m\cr
u_x(t,0)=0,\,\,  u(t,m)>0
\end{cases}
\end{equation}
for $t_1<t<t_2$,
where
$$
\begin{cases}
\tilde \beta(t,x)=-\chi_1 v_{1,x}(t,x;u_0,v_0)+\chi_2 v_{2,x}(t,x;u_0,v_0)\cr
\tilde a(t,x)=a(t,x)-\chi_1\lambda_1 v_1(t,x;u_0,h_0)+
\chi_2\lambda_2 v_2(t,x;u_0,h_0)\cr
\tilde b(t,x)=b(t,x)-\chi_1\mu_1+\chi_2\mu_2.
\end{cases}
$$

Consider
\begin{equation}
\label{aux-thm2-eq2}
\begin{cases}
w_t=w_{xx}+ \beta_0 w_x+w(\frac{a_{\inf}}{2}-(b_{\sup}-\chi_1\mu_1+\chi_2\mu_2)w),\quad 0<x<m\cr
w_x(t,0)=w(t,m)=0\cr
w(t_1,x)=\tilde u_0(x),\quad 0\le x\le m,
\end{cases}
\end{equation}
where
$$
\beta_0=\sup_{t_1\le t\le t_2,0\le x\le m} \tilde \beta(t,x)
$$
and
$\tilde u_0(x)=\delta_\epsilon \cos(\frac{\pi x}{2m})$. By {\bf Claim 1},
$$
\tilde u_0(x)\le u(t_1,x;u_0,h_0),\quad 0\le x\le m.
$$
By Lemma \ref{thm2-lm1},
$$
0\le  \beta_0\ll 1,\quad \tilde a(t,x)\ge \frac{a_{\inf}}{2}\quad \forall\,\,  0<\epsilon\ll 1.
$$
Then by Lemma \ref{fisher-kpp-lm1}, there is $\tilde \epsilon>0$ such that
$$
\|w(t,\cdot;t_1,\tilde u_0)\|_{C([0,m])}>\tilde \epsilon\quad \forall\,\, t>t_1 ,
$$
where $w(t,\cdot;t_1,\tilde u_0)$ is the solution  of \eqref{aux-thm2-eq2}.
Note that $\tilde u_0'(x)\leq 0 $, by comparison principle for parabolic equations,
$$
w_x(t,x;t_1,\tilde u_0)\le 0\quad \forall\,\, t>t_1,\, 0\le x\le m
$$
and then
$$
\tilde \beta(t,x) w_x(t,x;t_1,\tilde u_0)\ge \beta_0 w_x(t,x;t_1,\tilde u_0)\quad \forall\,\, t_1\le t\le t_2,\,\, 0\le x\le m.
$$
By comparison principle for parabolic equations again,
$$
u(t,x;u_0,h_0)\ge w(t,x;t_1,\tilde u_0)\quad \forall\,\, t_1\le t\le t_2,\,\, 0\le x\le m.
$$
This implies that the claim holds.

\smallskip
We now prove

\smallskip

\noindent {\bf Claim 3.} {\it
$$
\liminf_{t\to\infty} \inf_{0\le x\le m}u(t,x;u_0,h_0)>0.
$$}

\smallskip

Choose $0<\epsilon\ll \|u(T_0,\cdot;u_0,h_0)\|_{C([0,h(T_0;u_0,h_0)])}$. Let
$$
I(\epsilon)=\{t>T_0\,|\, \|u(t,\cdot;u_0,h_0)\|_{C([0,h(t;u_0,h_0)]}<\epsilon\}.
$$
Then $I(\epsilon)$ is an open subset of $(T_0,\infty)$ and $I(\epsilon)\not = (T_0,\infty)$.
Therefore, if $I(\epsilon)\not =\emptyset$, there are $T_0<s_n<t_n$ such that
$$
I(\epsilon)=\cup_{n} (s_n,t_n).
$$
By Claim 1, for any $t\in [T_0,\infty)\setminus I(\epsilon)$,
$$
\inf_{0\le x\le m} u(t,x;u_0,h_0)\ge \delta_\epsilon.
$$
For any $t\in I(\epsilon)$, there is $n$ such that $t\in (s_n,t_n)$. Note that
$\|u(s_n,\cdot;u_0,h_0)\|_{C([0,h(s_0;u_0,h_0)])}=\epsilon$. Then by Claims 1 and 2,
$$
\inf_{0\le x\le m} u(t,x;u_0,h_0)\ge \delta_{\tilde \epsilon}.
$$
It then follows that
$$
\liminf_{t\to\infty}\inf_{0\le x\le m} u(t,x;u_0,h_0)>0
$$
and the claim is proved. (ii) thus follows.
\end{proof}

\section{Persistence and convergence}

In this section, we study the local uniform persistence and convergence of nonnegative solutions of \eqref{one-free-boundary-eq}
in the case that spreading occurs, and prove Theorem \ref{free-boundary-thm3}. We first prove a lemma.

\begin{lem}
\label{thm3-lm1}
Assume that {\bf (H1)} holds.
 Let $h_0>0$ and $u_0(\cdot)$ satisfy \eqref{initial-cond-eq}, and $(u(\cdot,\cdot;u_0,h_0)$, $v_{1}(\cdot,\cdot;u_0,h_0)$, $v_2(\cdot,\cdot;u_0,h_0)$, $h(t;u_0,h_0))$ be the classical solution of \eqref{one-free-boundary-eq} with $u(0,\cdot;u_0,h_0)=u_0(\cdot)$
 and $h(0;u_0,h_0)=h_0$.
For every $R\gg 1$, there are $C_R\gg 1$ and $\varepsilon_R>0$ such that for any $t>0$ with $h(t;u_0,h_0)>2R$ and any $x_0\in (R,h(t;u_0,h_0)-R)$,  we have
 \begin{align}\label{aux-thm3-eq1}
&|\chi_i v_{i,x}(t,\cdot;u_0,h_0)|_{C([x_0-\frac{R}{2},x_0+\frac{R}{2}])} + |\chi_i\lambda_i v_i(t,\cdot;u_0,h_0)|_{C([x_0-\frac{R}{2}.x_0+\frac{R}{2}])}\nonumber\\
&\leq C_{R}\|u(t,\cdot;u_0,h_0)\|_{C([x_0-R,x_0+R])}+\varepsilon_R \max\{\|u_0\|_{C([0,h_0])},M_0\},\quad i=1,2
\end{align}
with $\lim_{R\to\infty}\varepsilon_R=0$,  where $M_0$ is as in \eqref{M0-eq}.
\end{lem}

\begin{proof}
Similar to Lemma \ref{thm2-lm2}, it  follows from \eqref{thm2-eq1} and  \cite[Lemma 2.5]{SaShXu}.
\end{proof}

\begin{proof}[Proof of Theorem \ref{free-boundary-thm3}]

(i) We prove the case that $t_0=0$. For given $h_0>0$ and $u_0(\cdot)$ satisfying \eqref{initial-cond-eq},
let $(u(t,x;u_0,h_0),v_1(t,x;u_0,h_0),v_2(t,x;u_0,h_0),h(t;u_0,h_0))$ be the solution
of \eqref{one-free-boundary-eq} with $u(0,x;u_0,h_0)=u_0(x)$ and $h(0;u_0,h_0)=h_0$.
Assume that $h_\infty:=\lim_{t\to\infty} h(t;u_0,h_0)=\infty$.

Fix $0<\epsilon \ll 1$, choose $R\gg 1$ such that
$$
2\varepsilon_R \max\{\|u_0\|_{C([0,h_0])},M_0\}<\frac{a_{\inf}}{3}.
$$
For any $x_0\in (R,\infty)$, there is $T_0>0$ such that
$$
h(T_0;u_0,h_0)>x_0+R.
$$
We first prove

\smallskip

\noindent{\bf Claim 1.} {\it For any $0<\epsilon\ll 1$, there is $\delta_\epsilon>0$ such that for any $t\ge T_0$, if $\sup_{x_0-R\le x\le x_0+R} u(t,x$; $u_0,h_0)=\epsilon$,
then $\inf_{x_0-R\le x\le x_0+R} u(t,x;u_0,h_0)\ge \delta_\epsilon$.
}

\smallskip

It can be proved by the similar arguments as those in the claim 1 of Theorem \ref{free-boundary-thm2}.

\smallskip

Next, we prove

\smallskip

\noindent {\bf Claim 2.} {\it  $\limsup_{t\to\infty} \sup_{x_0-R\le x\le x_0+R}u(t,x;u_0,h_0)\ge \epsilon$.}

\smallskip

In fact, assume that $\limsup_{t\to\infty} \sup_{x_0-R\le x\le x_0+R} u(t,x;u_0,h_0)< \epsilon$. Then there is $T_1>T_0$ such that
$$
u(t,x;u_0,h_0)<\epsilon \quad \forall\,\, t\ge T_1,\,\, x_0-R\le x\le x_0+R.
$$
Note that $u(t,x;u_0,h_0)$ satisfies
\begin{equation}
\label{aux-thm3-eq2}
\begin{cases}
u_t=u_{xx}+ \beta(t,x)u_x+u(\tilde a(t,x)-\tilde b(t,x)u),\quad x_0-\frac{R}{2}<x<x_0+\frac{R}{2}\cr
u(t,x_0-\frac{R}{2})>0,\,\,  u(t,x_0+\frac{R}{2})>0
\end{cases}
\end{equation}
for $t>T_1$,
where
$$
\begin{cases}
 \beta(t,x)=-\chi_1 v_{1,x}(t,x;u_0,v_0)+\chi_2 v_{2,x}(t,x;u_0,v_0)\cr
\tilde a(t,x)=a(t,x)-\chi_1\lambda_1 v_1(t,x;u_0,h_0)+
\chi_1\lambda_2 v_2(t,x;u_0,h_0)\cr
\tilde b(t,x)=b(t,x)-\chi_1\mu_1+\chi_2\mu_2.
\end{cases}
$$

Consider
\begin{equation}
\label{aux-thm3-eq3}
\begin{cases}
w_t=w_{xx}+ \beta(t,x) w_x+w(\frac{a_{\inf}}{2}-(b_{\sup}-\chi_1\mu_1+\chi_2\mu_2)w),\quad x_0-\frac{R}{2}<x<x_0+\frac{R}{2}\cr
w(t,x_0-\frac{R}{2})=w(t,x_0+\frac{R}{2})=0\cr
w(T_1,x)=\tilde u_0(x),\quad x_0-\frac{R}{2}\le x\le x_0+\frac{R}{2},
\end{cases}
\end{equation}
where
$\tilde u_0(x)=\inf_{x_0-R\leq x\leq x_0+R}u(T_1,x;u_0,h_0)\cdot  \cos (\frac{\pi(x-x_0)}{R})$. Note that
$$
\tilde u_0(x)\le u(T_1,x;u_0,h_0),\quad x_0-\frac{R}{2}\le x \le x_0+\frac{R}{2}.
$$
By Lemma \ref{thm3-lm1}, when $0<\epsilon\ll 1$,
$$
|\beta(t,x)|\ll 1\quad \forall\,\, t\ge T_1,\,\, x_0-\frac{R}{2}\le x\le x_0+\frac{R}{2}.
$$
Then by Lemma \ref{fisher-kpp-lm2},
$$
\liminf_{t\to\infty} \|w(t,x;T_1,\tilde u_0)\|_{C([x_0-\frac{R}{2},x_0+\frac{R}{2}])}\ge \epsilon,
$$
where $w(t,x;T_0,\tilde u_0)$ is the solution of \eqref{aux-thm3-eq3}. By comparison principle for parabolic equations,
$$
u(t,x;u_0,h_0)\ge w(t,x;T_1,\tilde u_0)\quad \forall\,\, t\ge T_1,\,\, x_0-\frac{R}{2}\le x\le x_0+\frac{R}{2}
$$
provided $0<\epsilon\ll 1$. It then follows that
$$
\liminf_{t\to\infty} \sup_{x_0-R\le x\le x_0+R}u(t,x;u_0,h_0)\ge \epsilon,
$$
which is a contradiction. Hence the claims holds.

\smallskip

We then prove

\smallskip

\noindent {\bf Claim 3.} {\it For $0<\epsilon\ll 1$, there is $\tilde \epsilon>0$ such that
if $\sup_{x_0-R\le x\le x_0+R}u(t,x;u_0,h_0)\le \epsilon$ for $T_0\le t_1\le t\le t_2$ and
$\sup_{x_0-R\le x\le x_0+R} u(t_1,x;u_0,h_0)=\epsilon$, then
$$\liminf_{t\to\infty} \sup_{x_0-R\le x\le x_0+R} u(t,x;u_0,h_0)\ge \tilde \epsilon.
$$}

\smallskip

By the arguments of Claim 2,
$$
u(t,x;u_0,h_0)\ge w(t,x;t_1,\tilde u_0)\quad \forall\, \, t_1\le t\le t_2,\,\, x_0-\frac{R}{2}\le x\le x_0+\frac{R}{2},
$$
where $\tilde u_0(x)=\delta_\epsilon \cos(\frac{\pi(x-x_0)}{R})$.
It then follows that there is $\tilde \epsilon>0$ such that the claim holds.

\smallskip

Now we prove

\smallskip

\noindent{\bf Claim 4.} {\it There is $\tilde m_0>0$ such that for any $L>0$,
$\liminf_{t\to\infty} \inf_{0\le x\le L} u(t,x;u_0,h_0)\ge \tilde m_0$.}

\smallskip

In fact, choose $0<\epsilon\ll 1$. For any $x_0>R$, by Claim 2, there is $T_1>T_0>0$ such that
$h(T_0;u_0,h_0)>x_0+R$ and
$$
\sup_{x_0-R\le x\le x_0+R}u(T_1,x;u_0,h_0)=\epsilon.
$$
Let
$$
I(\epsilon)=\{t>T_1\,|\, \sup_{x_0-R\le x\le x_0+R}u(T_1,x;u_0,h_0)<\epsilon\}.
$$
Then $I(\epsilon)$ is an open subset of $(T_0,\infty)$. By Claims 1 and 3,
$$
\inf_{x_0-\frac{R}{2}\le x\le x_0+\frac{R}{2}}u(t,x;u_0,h_0)\ge \min\{\delta_\epsilon, \delta_{\tilde \epsilon}\}
\quad \forall t\ge T_1.
$$
By Theorem \ref{free-boundary-thm2}(ii),
$$
\liminf_{t\to\infty} \inf_{0\le x\le R}u(t,x;u_0,h_0)>0.
$$
Let
$$
\tilde m_0=\min\{\delta_\epsilon,\delta_{\tilde \epsilon},\liminf_{t\to\infty} \inf_{0\le x\le R}u(t,x;u_0,h_0)\}.
$$
We then have for any $L>0$,
$$
\liminf_{t\to\infty} \inf_{0\le x\le L}u(t,x;u_0,h_0)\ge \tilde m_0.
$$

Finally, we prove

\smallskip

\noindent {\bf Claim 5.} {\it For any $L>0$, $\liminf_{t\to\infty} \inf_{0\le x\le L}u(t,x;u_0,h_0)\ge  m_0$.}

\smallskip

For any given $t_n\to\infty$, without loss of generality, we may assume that
$u(t+t_n,x;u_0,h_0)\to u^*(t,x)$, $v_i(t+t_n,x;u_0,h_0)\to v_i^*(t,x)$, and $a(t+t_n,x)\to a^*(t,x)$, $b^*(t+t_n,x)\to b^*(t,x)$
as $n\to\infty$ locally uniformly. Note that $\inf_{t\in\RR,0\le x<\infty} u^*(t,x)\ge \tilde m_0$ and
 $(u^*(t,x),v_1^*(t,x),v_2^*(t,x))$ is an entire solution of
$$
\begin{cases}
u_t=u_{xx}-\chi_1(uv_{1,x})_x+\chi_2(uv_{2,x})_x+u(a^*(t,x)-b^*(t,x)u),\quad 0<x<\infty\cr
0=v_{1,xx}-\lambda_1 v_1+\mu_1 u,\quad 0<x<\infty\cr
0=v_{2,xx}-\lambda_2 v_2+\mu_2 u,\quad 0<x<\infty\cr
u_x(t,0)=v_{1,x}(t,0)=v_{2,x}(t,0)=0.
\end{cases}
$$
Then by Theorem \ref{half-line-thm}(2)(ii),
\begin{equation}\label{Persistence-U}
u^*(t,x)\ge  m_0\quad \forall\,\, t\in\RR,\, 0\le x<\infty.
\end{equation}
This implies that the claim holds true.
Theorem \ref{free-boundary-thm3} (i) is thus proved.

\medskip

(ii) We prove that for any $L>0$,
\begin{equation}
\label{Eq-asymptotic}
\lim_{t\to\infty} \sup_{0\le x\le L} |u(t,x;u_0,h_0)-u^*(t,x;a,b)|=0.
\end{equation}
We prove it by contradiction. If \eqref{Eq-asymptotic} does not exist, there must be a constant $L_0>0, t_n\to \infty$ such that
\begin{equation*}
 |u(t + t_n, x_n;u_0,h_0) - u^*(t + t_n,x_n;a,b)| \geq \epsilon_0,\quad x_n \in [0,L_0].
\end{equation*}
Without loss of generality, assume  that
$u(t+t_n,x;u_0,h_0)\to \tilde u^*(t,x)$, $v_i(t+t_n,x;u_,h_0)\to \tilde v_i^*(t,x)$, $u^*(t+t_n,x;a,b)\to \tilde u^{**}(t,x)$,
$v_i^*(t,x;a,b)\to \tilde v_i^{**}(t,x)$,  and
$a(t+t_n,x)\to \tilde a(t,x)$, $b(t+t_n,x)\to \tilde b(t,x)$ as $n\to\infty$ locally uniformly (see \cite[Theorem 1.3]{BaoShen1}).
Note that by \eqref{Persistence-U}, $\inf_{t\in\RR,0<x<\infty}\tilde u^*(t,x)>0$ and $\inf_{t\in\RR,0<x<\infty}\tilde u^{**}(t,x)>0$,
and $\tilde u^*(0,x)\neq\tilde u^{**}(0,x)$.  This implies that \eqref{half-line-eq2} has two strictly positive entire solutions
$(\tilde u^*(t,x),\tilde v_1^*(t,x),\tilde v_2^*(t,x))$ and $(\tilde u^{**}(t,x),\tilde v_1^{**}(t,x)$, $\tilde v_2^{**}(t,x))$,
which is a contradiction. Theorem \ref{free-boundary-thm3}(ii) is thus proved.

\smallskip

(iii) By Theorem \ref{half-line-thm}(4)(i),  for any $(\tilde a,\tilde b)\in H(a,b)$, \eqref{half-line-eq2} has a unique strictly positive solution $(u^*(t;\tilde, \tilde b),\frac{\mu}{\lambda} u^*(t;\tilde a,\tilde b))$, where $u(t;\tilde a,\tilde b)$ is the unique strictly
positive solution of the ODE
$$
u^{'}=u(\tilde a(t)-\tilde b(t) u).
$$
(iii) then follows from (ii).
\end{proof}

\section{Chemotaxis models with double free boundaries}

In this section, we consider the spreading-vanishing dynamics of \eqref{two-free-boundary-eq} and give an outline of the proof
of Theorem \ref{free-boundary-thm4}.

\begin{proof} [Outline of the proof of Theorem \ref{free-boundary-thm4}]
Without loss of generality, we assume that $t_0=0$.

Observe that we have either $h_\infty - g_\infty <\infty$ or $h_\infty - g_\infty = \infty$.
\smallskip

 Suppose that $h_\infty - g_\infty <\infty$.
 By the similar arguments as those in Theorem \ref{free-boundary-thm2}(i), we must have $h_\infty - g_\infty \leq l^{**}$ and $\|u(t,\cdot;u_0,g_0,h_0)\|_{C([g(t;u_0,g_0,h_0),h(t;u_0,g_0,h_0)])}\to 0$ as $t\to\infty$, which implies (i).

\smallskip

 Suppose that $h_\infty - g_\infty = \infty$. We prove that (ii) holds.
 We first claim that $h_\infty = -g_\infty = \infty.$ If the claim does not hold, without loss of generality, we may assume that $-\infty < g_\infty < h_\infty = \infty.$ By the similar arguments as those in Theorem \ref{free-boundary-thm2}(1), we have $g'(t)\to 0$ as $t\to \infty$.
 by the arguments of Theorem \ref{free-boundary-thm3}(i), for any $g_\infty<l_-<l_+$ with $l_+-l_-\gg 1$,
\begin{equation*}
 \liminf_{t\to\infty}  \inf_{l_-\le x\le l_+} u(t,x;u_0,h_0,g_0) >0.
\end{equation*}
Let $t_n\to\infty$ be such that $a(t + t_n,x)\to a^*(t,x)$, $b(t +t_n,x) \to  b^*(t,x )$, and $u(t+t_n,x;u_0,h_0,g_0) \to u^*(t,x)$,
$v_i(t+t_n,x;u_0,g_0,h_0)\to v^*_i(t,x)$ as $n\to\infty$ locally uniformly. Then $u^*(t,x)$ is the solution of
\begin{equation*}
\begin{cases}
 u_t = u_{xx}-\chi_1( u v_{1,x}^*)_x+\chi_2 (u v_{2,x}^*)_x+ u(a^*(t,x) - b^*(t,x)u),\quad g_\infty < x< \infty,
 \\
 u(t,g_\infty) = 0,
\end{cases}
\end{equation*}
and $\inf_{x\in [l_-,l_+]} u^*(t,x) >0.$ Then by Hopf Lemma for parabolic equations,
\begin{equation*}
 u_x^*(g_\infty,t) > 0.
\end{equation*}
This implies that
\begin{equation*}
 g'(t + t_n;u_0,h_0,g_0)\to -\nu u_x^*(g_\infty,t) < 0,
\end{equation*}
which contradicts to the fact that $g'(t;u_0,h_0,g_0)\to 0$ as $t\to\infty.$ Hence $(g_\infty,h_\infty) = (-\infty,\infty)$.

 By the similar argument as those in Theorem \ref{free-boundary-thm2}(ii), we have
\begin{equation*}
\liminf_{t\to\infty} \inf_{|x|\le L} u(t,x;u_0,h_0,g_0)>0\quad \forall\,\, L>0.
\end{equation*}

The remaining of (ii) follows from the arguments of Theorem \ref{free-boundary-thm3}(ii).
\end{proof}

\section*{Acknowledgement}
The authors would also like to thank the anonymous referee for valuable comments and suggestions which improved the presentation of this paper considerably.


\begin{thebibliography}{99}


\bibitem{Amann1995linear} H.  Amann.
 Linear and quasilinear parabolic problems, Vol. I. Abstract Linear Theorey, Monongraphs in Mathematics, 89, Birkhuser Boston, Inc. Boston, MA, 1995.

 \bibitem{BaoShen1} L. Bao and W. Shen, Logistic type attraction-repulsion chemotaxis systems with a free boundary or unbounded boundary.
I. Asymptotic dynamics in fixed unbounded domain, preprint.



%

\bibitem{Chen2000Afree}
X.~F. Chen and A.~Friedman,
\newblock { A free boundary problem arising in a model of wound healing}.
\newblock {\em SIAM J. Math. Anal.}, 32(4), 778--800, 2000.

%
%


\bibitem{DuLi} Y.-H. Du and Z.-G. Lin,
\newblock Spreading-vanishing dichotomy in the diffusive logistic model with a free
boundary.
\newblock{\em SIAM J. Math. Anal.}, 42, 377-405, 2010.



\bibitem{du2013diffusive}
Y.-H. Du, Z.-M. Guo, and R. Peng,
\newblock A diffusive logistic model with a free boundary in time-periodic
  environment.
\newblock {\em Journal of Functional Analysis}, 265(9):2089--2142, 2013.

\bibitem{du2013pulsating}
Y.-H. Du and X. Liang,
\newblock Pulsating semi-waves in periodic media and spreading speed determined by a free boundary model.
\newblock {\em Ann. Inst. H. Poincar$\acute{e}$ Anal. Non Lin$\acute{e}$aire}, 32(2) 279-305, 2015.




%
%

%

%






\bibitem{HeSh} G. Hetzer and W. Shen,  Uniform persistence, coexistence, and extinction in almost periodic/nonautonomous competition diffusion systems, {\it  SIAM J. Math. Anal.} 34 (1), 204-227, 2003.





%

\bibitem{Issa-Shen-2017}
Tahir B. Issa and W.~Shen,
\newblock Dynamics in chemotaxis models of parabolic-elliptic type on bounded domain with time and space dependent logistic sources,
\newblock {\em  SIAM J. Appl. Dyn. Syst.}, 16 (2), 926-973, 2017.

%

%
%
%
%
%
%
%
%
\bibitem{Ladyzenskaja1968linear}
O.A. Ladyzenskaja and V.A. Solonnikov and N.N. Ural'ceva,
\newblock {{Linear and Quasilinear Equations of Parabolic Type, Translation of Mathematical Monographs}},
\newblock Amer. Math. Soc. Transl., vol. 23, Amer. Math. Soc., Providence, RI, 1968.

\bibitem{Li2016duffusive}
F. Li, X. Liang, and W. Shen,
\newblock Diffusive KPP equations with free boundaries in time almost periodic environments: I. Spreading and vanishing dichotomy, \newblock {\em Discrete Contin. Dyn. Syst},36(6), 3317-3338, 2016.

\bibitem{Li2016duffusive2}
F. Li, X. Liang, and W. Shen,
\newblock Diffusive KPP equations with free boundaries in time almost periodic environments: II. Spreading speeds and semi-wave solutions,
\newblock{\em J. Differential Equations}, 261 (4), 2403-2445, 2016.



%

\bibitem{luca2003chemotactic}
M.~Luca, A.~Chavez-Ross, L.~Edelstein-Keshet, and A.~Mogilner,
\newblock Chemotactic signaling, microglia, and Alzheimer's disease senile plaques: Is there a connection?
\newblock{\em Bulletin of mathematical biology}, 65(4), 693-730,2003.

\bibitem{MiSh1} J. Mierczynski and W. Shen, Lyapunov exponents and asymptotic dynamics in random
kolmogorov models, {\it J. Evolution Equations}, 4, 377-390, 2004.



\bibitem{MiSh2} J. Mierczynski and W. Shen, Spectral Theory for Random and Nonautonomous Parabolic
Equations and Applications, Chapman \& Hall/CRC Monogr. Surv. Pure Appl. Math.,
Chapman \& Hall/CRC, Boca Raton, FL, 2008.





\bibitem{Salako2016spreading}
R.~B. Salako and W. Shen,
\newblock {Spreading Speeds and Traveling waves of a parabolic-elliptic chemotaxis system with logistic source on $\mathbb{R}^N$},
\newblock {\em  Discrete Contin. Dyn. Syst.},  37(12), 6189--6225, 2017.



\bibitem{Salako2017Global}
R.~B. Salako and W. Shen,
\newblock {Global classical solutions, stability of constant equilibria, and spreading speeds in attraction-repulsion chemotaxis systems with logistic source on $\mathbb{R}^N$},
\newblock {\em  Journal of Dynamics and Differential Equations}, to appear.


\bibitem{SaShXu}
R.~B. Salako, W. Shen, and S.~W. Xue,
\newblock{Can chemotaxis speed up or slow down the spatial spreading in parabolic-elliptic chemotaxis  systems with logistic source?} preprint.

%
%

\bibitem{Tao2013competing}
Y.-S. Tao and Z.~A. Wang,
\newblock {Competing effects of attraction vs. repulsion in chemotaxis},
\newblock {\em Math. Models Methods Appl. Sci.} , 23, 1-36, 2013.

\bibitem{Tello2007chemotaxis}
J.~I. Tello and M.~Winkler,
\newblock {A chemotaxis system with logistic source},
\newblock {\em Communications in Partial Differential Equations},
  32, 849--877, 2007.

%
%
%
%
%
%
%
%

\bibitem{Wu2006elliptic}
Z.~Wu,J.~Yin,and C.~Wang
\newblock {\em {Elliptic and Parabolic Equations}},
\newblock World Scientific Publishing, 2006.

%
%
%

\bibitem{Zhang2018Afree}
 W.~Zhang, Z.~Liu, and L.~Zhou,
\newblock{A free boundary problem for an attraction-repulsion chemotaxis system},
\newblock{\em Boundary Value Problems},191, 2018.

%
\end{thebibliography}
\end{document}